\documentclass[10pt,twoside]{article}

\def\YEAR{\year}\newcount\VOL\VOL=\YEAR\advance\VOL by-1995

\makeatletter
\def\magnification{\afterassignment\m@g\count@}
\def\m@g{\mag=\count@\hsize6.5truein\vsize8.9truein\dimen\footins8truein}
\makeatother

\usepackage[utf8]{inputenc}             
\usepackage[T1]{fontenc}                
\usepackage{amssymb,amsmath,amsthm}
\usepackage{xspace}
\usepackage{graphicx}                   
\usepackage[french, english]{babel}            
\usepackage{color}                      
\usepackage[left=3cm, right=3cm]{geometry}
\usepackage[citecolor = blue]{hyperref}
\usepackage{}
\usepackage{dsfont}
\usepackage{geometry}
\usepackage{tikz-cd}

\usetikzlibrary{arrows,positioning} 

\tikzset{
    >=stealth',
    punkt/.style={
           rectangle,
           rounded corners,
           draw=black, very thick,
           text width=5em,
           minimum height=2em,
           text centered},
    pil/.style={
           ->,
           thick,
           shorten <=2pt,
           shorten >=2pt,}
}

\newcommand{\fou}{\mathcal{F}}
\newcommand{\eps}{\varepsilon}
\newcommand{\R}{\mathbb{R}}
\newcommand{\C}{\mathbb{C}}
\newcommand{\N}{\mathbb{N}}

\newcommand{\pfty}{+\infty}
\newcommand{\one}{\mathds{1}}
\newcommand{\man}{\mathcal{X}}
\newcommand{\zeroless}{\setminus\{0\}}
\newcommand{\calA}{\mathcal{A}}
\newcommand{\calJ}{\mathcal{J}}
\newcommand{\Rnzls}{\mathbb{R}^n\setminus\{0\}}
\newcommand{\Sym}{\textup{Sym}}

\newenvironment{prf}[1][]
{\vskip 2mm  {\it \bf Proof#1. }}{$\Box$ \vskip 2mm}
\newenvironment{notation}
{\vskip 2mm \textbf{Notation}:\vskip 2mm}{}
\newtheorem{theorem}{Theorem}
\numberwithin{theorem}{section}
\newtheorem{lemma}[theorem]{Lemma}
\newtheorem{proposition}[theorem]{Proposition}
\newtheorem{corollary}[theorem]{Corollary}
\newtheorem{definition}[theorem]{Definition}
\newtheorem{remark}{Remark}
\newtheorem{claim}{Claim}

\title{Weighted local Weyl laws for elliptic operators}
\author{Alejandro Rivera}
\date{\today}

\begin{document}
\maketitle

\vspace{2cm}

\begin{abstract}
Let $A$ be an elliptic pseudo-differential operator of order $m$ on a closed manifold $\man$ of dimension $n>0$, formally positive self-ajdoint with respect to some positive smooth density $d\mu_\man$. Then, the spectrum of $A$ is made up of a sequence of eigenvalues $(\lambda_k)_{k\geq 1}$ whose corresponding eigenfunctions $(e_k)_{k\geq 1}$ are $C^\infty$ smooth. Fix $s\in\R$ and define \
\[
K_L^s(x,y)=\sum_{0<\lambda_k\leq L}\lambda_k^{-s} e_k(x)\overline{e_k(y)}\, .
\]
We derive asymptotic formulae near the diagonal for the kernels $K_L^s(x,y)$ when $L\rightarrow +\infty$ with fixed $s$. For $s=0$, $K^0_L$ is the kernel of the spectral projector studied by H\"ormander in \cite{ho68}. In the present work we build on H\"ormander's result to study the kernels $K^s_L$. If $s<\frac{n}{m}$, $K_L^s$ is of order $L^{-s+n/m}$ and near the diagonal, the rescaled leading term behaves like the Fourier transform of an explicit function of the symbol of $A$. If $s=\frac{n}{m}$, under some explicit generic condition on the principal symbol of $A$, which holds if $A$ is a differential operator, the kernel has order $\ln(L)$ and the leading term has a logarithmic divergence smoothed at scale $L^{-1/m}$. Our results also hold for elliptic differential Dirichlet eigenvalue problems.
\end{abstract}

\vfill

\pagebreak

\tableofcontents

\vfill

\pagebreak

\section{Introduction}

The purpose of the present work is to compute pointwise asymptotics of the integral kernels of certain operators defined by functional calculus from either elliptic self-adjoint pseudo-differential operators on a closed manifold or an elliptic self-adjoint Dirichlet boundary problem. Stating the results in full generality requires some vocabulary from semi-classical analysis and some additional definitions. For this reason, we start by stating our results in the simpler case of elliptic self-adjoint differential operators on a closed manifold. The general case is presented in Section \ref{s.results}. This, of course, leads to some redundancy between different statements which we accept for the sake of accessibility and transparency of the main results.\\ 

Let $\man$ be a smooth compact manifold without boundary, of positive dimension $n>0$ and equipped with a smooth positive density $d\mu_\man$. Let $A$ be a positive elliptic differential operator on $\man$ of positive order $m$. By this we mean that in any local coordinate system $x=(x_1,\dots,x_n)$ on $\man$ defined on $U\subset\R^n$, $A$, acts on $C^\infty_c(U)$ as
\[
\sum_{0\leq |\alpha|\leq m} a_{\alpha}(x)\partial^\alpha
\]
where $\alpha\in\N^n$ and $a_\alpha\in C^\infty(\R^n)$ and for each $\xi\in\R^n\setminus\{0\}$, we have
\[
\sigma_A(x,\xi):=\sum_{|\alpha|=m}a_\alpha(x)\xi^\alpha>0\, .
\]
The function $\sigma_A$ is called the \textbf{principal symbol} of $A$ in these coordinates. It is well known (and easy to check) that the principal symbol of $A$ read in different coordinates pieces together as a smooth function on the complement of the zero section in $T^*\man$. We assume that $A$ is symmetric with respect to the $L^2$-scalar  product on $(\man,d\mu_\man)$. Then one can show (see Subsection \ref{ss.setting}) that $A$ has a unique self-adjoint extension whose spectrum is made up of a non-decreasing sequence $(\lambda_k)_{k\in\N}$ of positive eigenvalues diverging to $+\infty$ with smooth $L^2$-normalized eigenfunctions $(e_k)_{k\in\N}$ forming a Hilbert basis for $L^2(\man,d\mu_\man)$. For each $L\geq 0$, let $\Pi_L$ be the $L^2$ orthogonal projector on the space spanned by the eigenfunctions $e_k$ such that $\lambda_k\leq L$. Since this space is finite-dimensional, $\Pi_L$ has a smooth integral kernel $E_L\in C^\infty(\man\times\man)$. More explicitely,

\[
\forall (x,y)\in\man\times\man,\, E_L(x,y)=\sum_{\lambda_k\leq L}e_k(x)\overline{e_k(y)}\, .
\]

In \cite{ho68}, H\"ormander studied the behavior of this kernel on a neighborhood of the diagonal as $L\rightarrow+\infty$. Integrating $E_L$ over the diagonal he recovered the following estimate, also known as \textbf{Weyl's law}:

\[
  \textup{Card}\{k\in\N\ |\ \lambda_k\leq L\}\sim \frac{1}{(2\pi)^n}\int_\man\int_{\sigma_A(x,\xi)\leq 1}\widetilde{d_x\mu_\man}(\xi) d\mu_\man(x)\times L^{\frac{n}{m}}
\]

where $\widetilde{d_x\mu_\man}(\xi)$ is the density induced on $T^*_x\man$ by $d\mu_\man$. H\"ormander's result is stronger than the above estimates in two respects. First because the error term obtained is smaller than the ones known before and is sharp in all generality. Secondly, the result actually provides local information concerning the behavior of the kernel $E_L$ near the diagonal, which is why is sometimes called the \textbf{local Weyl law}. We will state this theorem in Subsection \ref{ss.hlwl} (see Theorem \ref{t.hormander}). In recent years, H\"ormander's local Weyl law has received a lot of attention because $E_L$ turns out to be the covariance of a certain Gaussian field on $\man$ defined as a random linear combination of eigenfunctions of $A$ (see for instance \cite{ber85}, \cite{ns09}, \cite{zeld09}, \cite{gawe14}, \cite{let14}, \cite{gawe15}, \cite{nyri15}, \cite{hzz151}, \cite{ns15},\cite{sawi} and \cite{casa17}). In \cite{ale161} we studied a natural variation of this random linear combination of eigenfunctions in dimension $n=2$ and observed a very different asymptotic behavior of the covariance function. Following this work, we are interested in studying more general random linear combinations of these eigenfunctions. To this end, it is essential to gather some information about the corresponding covariance function. The purpose of this article is to provide an asymptotic for these kernels similar to the one we have for $E_L$. For each $s\in\R$ we consider the kernel

\[
K_L^s(x,y)=\sum_{\lambda_k\leq L} \lambda_k^{-s} e_k(x)\overline{e_k(y)}\, .
\]

These kernels converge in distribution to the integral kernels of $A^{-s}$ as $L\rightarrow +\infty$ but diverge on the diagonal for small or negative values of $s$. The pointwise behavior of the limiting kernel on the diagonal, which is well defined for large values of $s$, has been studied for instance in \cite{see67} and \cite{sch86}. In \cite{sch86}, the author proved that, as a function of $s$, the limit admitted a meromorphic extension to the whole complex plane. We focus instead on a fixed $s$ for which the kernel diverges and study its pointwise divergence near the diagonal. We call these results \textbf{weighted local Weyl laws} by analogy with $E_L$ (which is just $K_L^0$) because of the weights $\lambda_n^{-s}$ on the terms of the sum defining $K^s_L$. As we shall see, the kernels $K_L^s$ experience a sudden change in their asymptotic behavior between the phases $s<\frac{n}{m}$ and $s=\frac{n}{m}$. All our results will be local so we take the liberty of omitting with the composition with the chart when writing functions on $\man$ in local coordinates. Our first result provides information when $s<\frac{n}{m}$.
\begin{theorem}\label{t.diff.1}
Assume that $s<\frac{n}{m}$. Fix $x_0\in\man$ and consider local coordinates $x=(x_1,\dots,x_n)$ for $\man$ centered at $x_0$ and defined on an open subset $U\subset\R^n$ such that $d\mu_\man$ agrees with the Lebesgue measure in these coordinates. Then for each $\alpha,\beta\in\N^n$, there exists $V\subset U$ an open neighborhood of $0$ such that, in these coordinates, we have the following estimates.

\begin{enumerate}
\item Uniformly for $w,x,y\in V$ and $L\geq 1$ we have $w+L^{-1/m}x, w+L^{-1/m}y\in U$ and

\begin{align*}
L^{s-(n+|\alpha|+|\beta|)/m}\partial^\alpha_x\partial^\beta_yK^s_L\left(w+L^{-1/m}x,w+L^{-1/m}y\right)=&\frac{1}{(2\pi)^n}\int_{\sigma_A(w,\xi)\leq 1}e^{i\langle \xi,x-y\rangle}\frac{(i\xi)^\alpha(-i\xi)^\beta}{\sigma_A(w,\xi)^s} d\xi\\
&+O\left(L^{-1/m}\ln\left(L\right)^\eta\right)
\end{align*}

where $\eta = 1$ if $s=(n+|\alpha|+|\beta|-1)/m$ and $0$ otherwise.

\item Let $\eps>0$. Then, uniformly for $x,y\in V$ such that $|x-y|>\eps$ and $L\geq 1$,

\[
L^{s-(n+|\alpha|+|\beta|)/m}\partial^\alpha_x\partial^\beta_y K^s_L(x,y)=O\left(L^{-1/m}\ln\left(L\right)^\eta\right)
\]
where $\eta = 1$ if $s=(n+|\alpha|+|\beta|-1)/m$ and $0$ otherwise.
\end{enumerate}
Here $|\alpha|=\alpha_1+\dots+\alpha_n$ and $\langle\cdot,\cdot\rangle$ is the Euclidean scalar product.
\end{theorem}
Note that the case where $s=0$ and $\alpha=\beta=0$ is Theorem 5.1 of \cite{ho68} (see Theorem \ref{t.hormander} and the discussion below for more details about this case). We prove Theorem \ref{t.diff.1} at the end of Section \ref{s.results}. Before stating the second result, we introduce the following notation. Firstly, $d\mu_\man$ defines a canonical dual density $\widetilde{d\mu_\man}$ on $T^*\man$. Around any $x\in\man$, there exist local coordinates in which $d\mu_\man$ corresponds to the Lebesgue measure. Then $\widetilde{d\mu_\man}$ is the unique density on $T^*\man$ who in these coordinates corresponds to the Lebesgue measure. For each $x\in\man$, let $S^*_x=\{\xi\in T^*_x\man\, |\, \sigma_A(x,\xi)=1\}$. Since $\sigma_A$ is $m$-homogeneous, $S^*_x$ is a smooth compact hypersurface of $T^*_x\man$ strictly star-shaped\footnote{More precisely, for each $\xi\in T^*_x\man\setminus\{0\}$, the ray $\{t\xi\, |\, t>0\}$ intersects $S^*$ exactly at $\sigma_A(x,\xi)^{-1/m}\xi$ and does so transversally.} around the origin and there exists a smooth density $d_x\nu$ on $S^*_x$ such that for each $u\in C^\infty_c(T^*_x\man)$,
\begin{equation}\label{e.nu}
\int_{T^*_x\man} u(\xi)\widetilde{d_x\mu_\man}(\xi)=\int_0^{+\infty}\int_{S^*_x} u(t\xi) d_x\nu(\xi) t^{n-1}dt\, .
\end{equation}
Our second result deals with the case where $s=\frac{n}{m}$. While Theorem \ref{t.diff.1} proves that the rate of growth of $K_L^s$ does not depend on $s$ for $s<n/m$, and that the main term depends continuously on $s$, the following result shows that this is not true for $s=n/m$. Indeed, while the first point is analogous to the results of Theorem \ref{t.diff.1}
, the second point is quite different (and requires additional tools).
\begin{theorem}\label{t.diff.2}
Assume that $s=\frac{n}{m}$. Fix $x_0\in\man$ and consider local coordinates $x=(x_1,\dots,x_n)$ for $\man$ centered at $x_0$ and defined on an open subset $U\subset\R^n$ such that $d\mu_\man$ agrees with the Lebesgue measure in these coordinates. Then, for each $\alpha,\beta\in\N^n$, there exists an open neighborhood $V\subset U$ of $0$ such that the following holds.

\begin{enumerate}
\item
\begin{itemize}
\item Assume that $(\alpha,\beta)\neq (0,0)$. In these coordinates, uniformly for $w,x,y\in V$ and $L\geq 1$, $w+L^{-1/m}x, w+L^{-1/m}y\in U$ and

\begin{align*}
L^{-(|\alpha|+|\beta|)/m}\partial^\alpha_x\partial^\beta_yK^s_L\left(w+L^{-1/m}x,w+L^{-1/m}y\right)=&\frac{1}{(2\pi)^n}\int_{\sigma_A(w,\xi)\leq 1}e^{i\langle\xi,x-y\rangle}\frac{(i\xi)^\alpha(-i\xi)^\beta}{\sigma_A(w,\xi)^{n/m}} d\xi\\
&+O\left(L^{-1/m}\ln\left(L\right)^\eta\right)\, .
\end{align*}

where $\eta = 1$ if $1=|\alpha|+|\beta|$ and $0$ otherwise.

\item Assume $(\alpha,\beta)\neq (0,0)$ and let $\eps>0$. Then, uniformly for $x,y\in V$ such that $|x-y|>\eps$,

\[
L^{-(|\alpha|+|\beta|)/m}\partial^\alpha_x\partial^\beta_y K^s_L(x,y)=O\left(L^{-1/m}\ln\left(L\right)^\eta\right)
\]
where $\eta = 1$ if $1=|\alpha|+|\beta|$ and $0$ otherwise.

\end{itemize}

\item
\begin{itemize}
\item Uniformly for $x,y\in V$ and $L\geq 1$, in these coordinates,
\[
K^s_L(x,y)=g_A(x,y)\left[\ln\left(L^{1/m}\right)-\ln_+\left(L^{1/m}|x-y|\right)\right]+O(1)
\]
where
\[
g_A(x,y)=\frac{1}{(2\pi)^n}\times\frac{\nu_x\left(S^*_x\right)+\nu_y\left(S^*_y\right)}{2}
\]
and $\ln_+(t)=\ln(t)\vee 0$.
\item There exists a symmetric bounded function $Q: U\times U\rightarrow\R$ such that, uniformly for $\kappa\geq 1$, $L\geq 1$ and $x,y\in V$ such that $|x-y|\geq \kappa L^{-1/m}$, in these coordinates,
\[
K_L^s(x,y)=-g_A(x,y)\ln\left(|x-y|\right)+Q(x,y)+O\left(\kappa^{-1/k}\right)
\]
where, if $n=1$ then $k=1$ and if $n\geq 2$ then $k=m$.
\end{itemize}
\end{enumerate}
Here $|\alpha|=\alpha_1+\dots+\alpha_n$ and $\langle\cdot,\cdot\rangle$ is the Euclidean scalar product.
\end{theorem}
This theorem (especially the second point) generalizes Theorem 3 of \cite{ale161}, which proved the second point in the case where  $s=1$, $\man$ was a closed surface (so $n=2$) with a Riemmanian metric and $A$ was the associated Laplacian (so $m=2$). The main challenge in the extension comes from the need to apply a generalized stationary phase formula on the level sets of the symbol. In \cite{ale161}, this is Proposition 23, where the traditional stationary phase formula applies directly. This general setting requires tools from singularity theory that are deployed in Section \ref{s.singularities}. The second point of Theorem \ref{t.diff.2} will follow from Theorem \ref{t.main.2} below. As is apparent, in Figure 1, the proof of this result is more complex than that of the others. We prove Theorem \ref{t.diff.2} at the end of Section \ref{s.results}.\\

\begin{corollary}\label{co.diff.1}
The Schwartz kernel $K\in\mathcal{D}'(\man\times\man)$ of $A^{-n/m}$ belongs to $L^1(\man\times\man)$.  Moreover, for each smooth distance function $d:\man\times\man\rightarrow\R$ on $\man$ there exists a bounded symmetric function $Q_{A,d}:\man\times\man\rightarrow\R$, smooth on the complement of the diagonal, such that, for any distinct $x,y\in\man$,
\[
K(x,y)=-g_A(x,y)\ln\left(d\left(x,y\right)\right)+Q_{A,d}(x,y)\, .
\]
\end{corollary}

We prove Corollary \ref{co.diff.1} at the end of Section \ref{s.results}.\\

\subsection{An important example: the Laplacian}

As explained above, this work is motivated by recent interest in the kernel $K^0_L$ as the covariance function of a Gaussian field. In further work, we wish to study certain Gaussian fields arising naturally in geometry and statistical mechanics with covariance $K^s_L$. One such field is the Gaussian Free Field, which is a central object in statistical mechanics today. In Corollaries \ref{c.gff.1} and \ref{c.gff.2} we detail our main results in this special case.\\

Let $(\man,g)$ be a closed Riemmanian manifold of dimension $n\geq 2$. Let $\Delta=-div\circ\nabla$ be the Laplace operator on $\man$ and let $|dV_g|$ be the Riemmanian volume density on $\man$. Then, $\Delta$ is an elliptic differential operator with principal symbol $\sigma(x,\xi)=g_x^{-1}\left(\xi,\xi\right)$ where $g^{-1}_x$ is the scalar product induced on $T^*_x\man$ by $g_x$. Moreover, $\Delta$ is symmetric with respect to the $L^2$-scalar product induced by the density $|dV_g|$ on $\man$. Let $\left(\lambda_k\right)_{k\in\N}$ be the sequence of eigenvalues of $\Delta$ (counted with multiplicity) and arranged in increasing order. Let $\left(e_k\right)_{k\in\N}$ be a Hilbert basis of $L^2\left(\man,|dV_g|\right)$ made up of real valued functions, such that for each $k\in\N$, $\Delta e_k=\lambda_k e_k$. For each $L>0$, each $s>0$ and each $(x,y)\in\man\times\man$, let
\[
K^s_L(x,y)=\sum_{0<\lambda_k\leq L}\lambda_k^{-1}e_k(x)e_k(y)\, .
\]
Then, $K^1_L$ converges in distribution as $L\rightarrow+\infty$ to the Green function on $\man$ which is the (generalized) covariance function for the Gaussian Free Field (see for instance \cite{shef_gff}). We have the following results. In the case where $s<n/2$, $K^s_L$ converges at scale $L^{-1/2}$ to a non-trivial function after rescaling by a polynomial factor.
\begin{corollary}\label{c.gff.1}
Assume that $s<n/2$. Fix $x_0\in\man$ and consider local coordinates $x=(x_1,\dots,x_n)$ for $\man$ centered at $x_0$ such that $|dV_g|$ agrees with the Lebesgue measure in these coordinates. Then, for each $\alpha,\beta\in\N^n$ there exists $V\subset U$ an open neighborhood of $0$ such that, in these coordinates, we have the following estimates.
\begin{enumerate}
\item Uniformly for $w,x,y\in V$ and $L\geq 1$ we have $w+L^{-1/2}x,w+L^{-1/2}y\in U$ and
\begin{multline*}
\partial_x^\alpha\partial_y^\beta K^s_L\left(w+L^{-1/2}x,w+L^{-1/2}y\right)=\frac{1}{(2\pi)^n}\int_{|\xi|_w^2\leq 1}e^{i\langle\xi,x-y\rangle}\frac{(i\xi)^\alpha(-i\xi)^\beta}{|\xi|_w^{2s}} d\xi L^{(n+|\alpha|+|\beta|-2s)/2}\\
+O\left(L^{(n+|\alpha|+|\beta|-2s-1)/2}\ln\left(L\right)^\eta\right)
\end{multline*}
where $\eta=1$ if $s=(n+|\alpha|+|\beta|-1)/2$ and $0$ otherwise. Here $|\xi|_w^2=g^{-1}_w(\xi,\xi)$ and $d\xi$ is the Lebesgue measure.
\item Let $\eps>0$. Then, uniformly for $x,y\in V$ such that $|x-y|>\eps$ and for $L\geq 1$,
\[
\partial_x^\alpha\partial_y^\beta K^s_L\left(w+L^{-1/2}x,w+L^{-1/2}y\right)=O\left(L^{(n+|\alpha|+|\beta|-2s-1)/2}\ln\left(L\right)^\eta\right)
\]
where $\eta=1$ if $s=(n+|\alpha|+|\beta|-1)/2$ and $0$ otherwise.
\end{enumerate}
\end{corollary}
\begin{prf}
This follows directly from Theorem \ref{t.diff.1} with $m=2$, $s<n/2$, $A=\Delta$ and $\sigma_A(w,\xi)=|\xi|_w^2$.
\end{prf}
On the other hand, if $s=n/2$, although the derivatives of $K^s_L$ also have non-trivial local limits at scale $L^{-1/2}$, $K^s_L$ itself converges pointwise to a distribution with a logarithmic singularity on the diagonal. Note that when $s=1$, the first part of the second point of Corollary \ref{c.gff.2} below yields Theorem 3 of \cite{ale161}.
\begin{corollary}\label{c.gff.2}
Assume that $s=n/2$. Fix $x_0\in\man$ and consider local coordinates $x=(x_1,x_2)$ for $\man$ centered at $x_0$ defined on an open subset $U\subset\R^2$ such that $|dV_g|$ agrees with the Lebesgue measure in these coordinates. Then, for each $\alpha,\beta\in\N^2$, there exists an open neighborhood $V\subset U$ of $0$ such that the following holds.
\begin{enumerate}
\item
\begin{itemize}
\item Assume that $(\alpha,\beta)\neq(0,0)$. In these coordinates, uniformly for $w,x,y\in V$ and $L\geq 1$, we have $w+L^{-1/2}x,w+L^{-1/2}y\in U$ and
\begin{multline*}
\partial^\alpha_x\partial^\beta_yK^s_L\left(w+L^{-1/2}x,w+L^{-1/2}y\right)=\frac{1}{(2\pi)^n}\int_{|\xi|_w^2\leq 1} e^{i\langle\xi,x-y\rangle}\frac{(i\xi)^\alpha(-i\xi)^\beta}{|\xi|_w^n}d\xi L^{(n+|\alpha|+|\beta|-2s)/2}\\
+O\left(L^{(n+|\alpha|+|\beta|-2s-1)/2}\ln(L)^\eta\right)
\end{multline*}
where $\eta=1$ if $1=|\alpha|+|\beta|$ and $0$ otherwise. Here $|\xi|_w^2=g_w^{-1}(\xi,\xi)$ and $d\xi$ is the Lebesgue measure.
\item Let $\eps>0$. Then, uniformly for $x,y\in V$ such that $|x-y|>\eps$ and for $L\geq 1$,
\[
\partial_x^\alpha\partial_y^\beta K^s_L\left(w+L^{-1/2}x,w+L^{-1/2}y\right)=O\left(L^{(n+|\alpha|+|\beta|-2s-1)/2}\ln\left(L\right)^\eta\right)
\]
where $\eta=1$ if $1=|\alpha|+|\beta|$ and $0$ otherwise.
\end{itemize}
\item
\begin{itemize}
\item Uniformly for $x,y\in V$ and $L\geq 1$, in these coordinates,
\[
G_L(x,y)=\frac{\left|S^{n-1}\right|}{(2\pi)^n}\left[\ln\left(L^{1/2}\right)-\ln_+\left(L^{1/2}|x-y|\right)\right]+O(1)
\]
where $\ln_+(t)=\ln(t)\vee 0$.
\item There exists a symmetric bounded function $Q:U\times U\rightarrow\R$ such that, uniformly for $\kappa\geq 1$, $L\geq 1$ and $x,y\in V$ such that $|x-y|\geq \kappa L^{-1/2}$, in these coordinates,
\[
G_L(x,y)=\frac{\left|S^{n-1}\right|}{(2\pi)^n}\ln(|x-y|)+Q(x,y)+O\left(\kappa^{-1/2}\right)\, .
\]
\end{itemize}
\end{enumerate}
\end{corollary}
\begin{prf}
This follows directly from Theorem \ref{t.diff.2} with $m=2$, $s=n/2$, $A=\Delta$ and $\sigma_A(w,\xi)=|\xi|_w^2$.
\end{prf}
\vspace{1cm}$ $\\
\textbf{Acknowledgements:}\\
I am grateful my advisor Damien Gayet, for suggesting that I tackle this question, but also for his helpful advice in the presentation, organization and revision of this manuscript and finally for taking the time to proof-read it. I would also like to thank Simon Andr\' eys for his inspiring explanations on singularity theory.

\section{Statement of the main results}\label{s.results}

In this section, we present the main objects of study and state our results in full generality. In Subsection \ref{ss.setting} we present the general framework of the article. In Subsection \ref{ss.hlwl} we state H\"ormander's local Weyl law. In Subsection \ref{ss.wlwl} we state the generalizations of the local Weyl law proved in this paper. We finish off by deducing Theorems \ref{t.diff.1} and \ref{t.diff.2} as well as Corollary \ref{co.diff.1}.\\

\subsection{General setting}\label{ss.setting}

In this article, we consider simultaneously two different elliptic eigenvalue problems. Since our arguments hold indifferently for the two cases, we present them in this section using the same notations. The first case is a closed eigenvalue problem. In this case we will follow \cite{ho68}. In the second case, we consider a Dirichlet eigenvalue problem, for which our main reference will be \cite{vas83}.

\begin{enumerate}
\item
In this setting we follow \cite{ho68}. Here $\man$ is a compact manifold without boundary. We consider a classical elliptic pseudo-differential operator $A$ of positive order $m$ acting on $C^\infty_c(\man)$. We assume $A$ is symmetric for the $L^2$-scalar product on $(\man,d\mu_\man)$. This implies that the principal symbol $\sigma_A$ of $A$ is real valued and positive homogeneous of order $m$. Moreover, under these assumptions, $A$ has a unique self-adjoint extension in $L^2(\man,d\mu_\man)$ whose spectrum forms a discrete non-decreasing sequence $(\lambda_k)_{k\in\N}$ of real numbers diverging to $+\infty$ and whose corresponding eigenfunctions $e_k$ are of class $C^\infty$ (see for instance Section 29.1 of \cite{ho_apdo4}). For each $L>0$, set

\[
\forall x,y\in\overline{\man}, E_L(x,y)=\sum_{\lambda_j\leq L}e_j(x)\overline{e_j(y)}\, .
\]

\item
In this setting, we follow \cite{vas83}. Here $\man$ is the interior of a compact manifold $\overline{\man}$ with non-empty boundary $\partial\man$. We assume that $\man$ is equipped with a positive density $d\mu_\man$. We consider the Dirichlet eigenvalue problem

\begin{align*}
Au&=\lambda u\text{ on }\man\, ;\\
\forall j=1,\dots,j_0,\, B_ju&=0\text{ on }\partial\man
\end{align*}

where $A$ is an elliptic differential operator of even order $m\geq 1$ with principal symbol $\sigma_A$, $B_j$ are boundary differential operators (see Chapter 2, Section 1.4 of \cite{lions_magenes_1}) and $\lambda\in\C$. We assume that the problem is elliptic, formally self-adjoint with respect to $d\mu_\man$ and semi-bounded from below (see \cite{vas83} Section 1). As is well known, under these assumptions, the values of $\lambda$ for which this problem has a non-trivial solution with sufficient regularity form an non-decreasing sequence $(\lambda_k)_{k\in\N}$ of real numbers diverging to $+\infty$ and the corresponding eigenfunctions $e_k$ are smooth in $\man$ up to the boundary (the proof goes along the same lines as in the closed case treated in Section 29.1 of \cite{ho_apdo4} and is easily adapted using results from Chapter 20 of \cite{ho_apdo3}). For each $L>0$, set
\[
\forall x,y\in\overline{\man}, E_L(x,y)=\sum_{\lambda_j\leq L}e_j(x)\overline{e_j(y)}\, .
\]
\end{enumerate}

\subsection{H\"ormander's local Weyl law}\label{ss.hlwl}

Let us consider the sequence of real numbers $(\lambda_k)_k$ and the sequence of smooth functions $(e_k)_k$ from either of the two settings presented in Subsection \ref{ss.setting}. Recall that $\sigma_A$ is the principal symbol of $A$, which we assumed to be positive homogeneous of order $m>0$ in the second variable. We begin by stating H\"ormander's local Weyl law, for which we need the following definition\footnote{This definition is inspired by Definition 2.3 of \cite{ho68}. However, our notion of proper phase function is more restrictive than H\"ormander's notion of phase function.}.
\begin{definition}\label{d.phase}
Given an open subset $U\subset\R^n$, we will say that a function $\psi\in C^\infty(U\times U\times\R^n)$ is a \textbf{proper phase function} if it satisfies the following conditions.
\begin{enumerate}
\item The function $\psi$ is a symbol of order one in its third variable.
\item For each $(x,y,\xi)\in U\times U\times\R^n$, $\langle x-y,\xi\rangle =0$ implies that $\psi(x,y,\xi)=0$.
\item For each $x\in U$ and $\xi\in\R^n$, $\partial_x\psi(x,y,\xi)|_{y=x}=\xi$.
\item There exists $\psi_\infty\in C^\infty(U\times U\times \R^n\zeroless)$ satisfying all of the above properties and $1$-homogeneous in $\xi$ such that
\[t^{-1}\psi(x,y,t\xi)\xrightarrow[t\to\pfty]{}\psi_\infty(x,y,\xi)\]
where the convergence takes place in $S^1(U\times U\times \R^n\zeroless)$. 
\end{enumerate}
\end{definition}
An important example of proper phase function to have in mind is the phase function $\psi(x,y,\xi)=\langle x-y,\xi\rangle$. H\"ormander's local Weyl law may be stated as follows.
\begin{theorem}[\cite{ho68}, Theorem 5.1 for $P=Id$]\label{t.hormander}
Let $P$ be a differential operator of order $d$ acting on $C^\infty(\man\times\man)$. Fix a point in $\man$ and consider local coordinates $(x_1,\dots,x_n)$ around it. Suppose further that the density $d\mu_{\man}$ agrees with the Lebesgue measure in these coordinates. Let $\sigma_A$ (resp. $\sigma_P$) be the principal symbol of $A$ (resp. $P$) in these coordinates. Then, there exists an open neighborhood $U$ of $0\in\R^n$, a proper phase function function $\psi\in C^\infty(U\times U\times\R^n)$ and a constant $C<+\infty$, such that, in these coordinates, for each $x,y\in U$ and $L>0$,

\[
\left|PE_L(x,y)-\frac{1}{(2\pi)^n}\int_{\sigma_A(x,\xi)\leq L}e^{i\psi(x,y,\xi)}\sigma_P(x,y,\partial_{x,y}\psi(x,y,\xi))d\xi\right|\leq C(1+L)^{(n+d-1)/m}\, .
\]

Moreover, for each neighborhood $W\subset U\times U$ of the diagonal there exists $C>0$ such that in local coordinates, for each $(x,y)\in \left(U\times U\right)\setminus W$ and $L>0$,

\[
\Big|PE_L(x,y)\Big|\leq C(1+L)^{(n+d-1)/m}\, .
\]

Finally, there exists a symbol $\sigma\in S^1$ such that $\sigma_A^{1/m}-\sigma\in S^0$ and for each $x,y\in U$ and $\xi\in\R^n$,
\begin{equation}\label{e.eikonal}
\sigma(x,\partial_x\psi(x,y,\xi))=\sigma(y,\xi)\, .
\end{equation}
\end{theorem}
Here $\partial_{x,y}\psi$ denotes the partial derivative of $\psi$ with respect to the couple $(x,y)$.
\begin{remark}
The asymptotic provided by Theorem \ref{t.main.2} is coordinate dependent since the notion of proper phase function is not invariant.
\end{remark}
\begin{remark}
Equation \eqref{e.eikonal} is called the \textbf{eikonal equation} and it has a unique solution with the boundary conditions imposed by the admissibility condition (see Section 3 of \cite{ho68}). The part concerning the eikonal equation is not usually stated as part of the local Weyl law but the function $\psi$ provided by the theorem does satisfy this property and it will be useful in our proofs.
\end{remark}
\begin{remark}
The case where $P=Id$ and was proved by H\"ormander in \cite{ho68}. The case where $x=y$ and $\man$ is a closed manifold was treated in \cite{sava_adep} with some restrictions on $P$. Finally, Gayet and Welschinger extended this result to a general $P$ (see Theorem 2.3 of \cite{gawe14}) on a closed manifold. While in their statement, $x=y$, their proof yields the off-diagonal case with only minor modifications.
\end{remark}
\begin{remark}
H\"ormander manages to lift the compactness assumption using results on the local nature of the spectral projector $\Pi_L$. It is not clear that this approach could be applied for a general $P$.
\end{remark}
\begin{remark}
Notice that in the boundary problem case (as in setting 2. form Subsection \ref{ss.setting}) we only get estimates in the interior of the domain.
\end{remark}
\begin{remark}
One recent result closely related to this theorem is Canzani and Hanin's asymptotics for the monochromatic spectral projector of the Laplacian under some dynamical assumption on the geodesic flow (see \cite{caha151} and \cite{caha16}).
\end{remark}
For the convenience of the reader, in Appendix \ref{s.hormander} we provide a proof of the full result relying on the wave kernel asymptotics provided in \cite{ho68}.\\

\subsection{Weighted local Weyl laws}\label{ss.wlwl}

In the present article, we generalize Theorem \ref{t.hormander} in the following way. Consider $A$ and $P$ as in Theorem \ref{t.hormander} and take $U$ and $\psi$ as provided by this theorem.
\begin{theorem}\label{t.main.1}
Fix $z=z_1+iz_2\in\C$. Let $f\in C^\infty(\R)$ such that $f(t)=t^z$ for $t$ large enough. Let $K_L$ be the Schwartz kernel of $\Pi_Lf(A)$. Suppose that $n+d+mz_1>0$. For each $x,y\in U$ and $L\geq 1$, let
\[R_L(x,y)=L^{-z_1-(n+d)/m}\Big[PK_L(x,y)-\frac{1}{(2\pi)^n}\int_{\sigma_A(0,\xi)\leq 1}e^{i\langle \xi,x-y\rangle L^{1/m}}\sigma_A(0,\xi)^z\sigma_P(0,0,(\xi,-\xi))d\xi\Big].\]
Then, there exists an open neighborhood $V$ of $0\in U$ such that the following holds.
\begin{enumerate}
\item Uniformly for $L\geq 1$ and $(w,x,y)\in V\times V\times V$, $w+L^{-1/m}x, w+L^{-1/m}y\in U$ and
\[
R_L\left(w+L^{-1/m}x,w+L^{-1/m}y\right)=O\left(L^{-1/m}\ln(L)^\eta\right)
\]
where $\eta=1$ if $n+d+mz=1$ and $0$ otherwise.
\item Uniformly for $L\geq 1$ and $(x,y)\in V\times V\setminus W$,
\[
PK_L(x,y)=O\left(L^{z_1+(n+d-1)/m}\ln(L)^\eta\right)
\]
where $\eta=1$ if $n+d+mz=1$ and $0$ otherwise.
\end{enumerate}
\end{theorem}
We prove Theorem \ref{t.main.1} in Section \ref{s.proofs.1}. As we will see below, Theorem \ref{t.diff.1} and the first assertion of Theorem \ref{t.diff.2} are both direct consequences of Theorem \ref{t.main.1}. Before stating Theorem \ref{t.main.2} below, we must introduce some more terminology. One key ingredient of the proof will be the decay of certain oscillatory integrals depending on the level sets of $\sigma_A$. To observe this behavior we must impose certain condition on $\sigma_A$. This is the object of Definition \ref{d.admissible}.
\begin{definition}\label{d.admissible}
Fix $m\in\R$, $m>0$ and $k_0\in\N$, $k_0\geq 2$. We say that a positive $m$-homogeneous symbol $\sigma$ on $U$ is $k_0$-admissible there exists $k_0\geq 2$ such that
\begin{equation}\label{e.admissible}
\forall (x,\xi)\in U\times(\R^n\setminus \{0\})\ \exists k\in\{2,\dots,k_0\},\ \sigma(x,\xi)^{k-1}\partial_\xi^k\sigma(x,\xi)\neq \frac{m(m-1)\dots(m-k+1)}{m^k}(\partial_\xi\sigma(x,\xi))^{\otimes k}.
\end{equation}
\end{definition}
This condition is invariant if we see $\sigma$ as a function on $T^*\man$ because coordinate changes act linearly on the fibers of $T^*\man$. It is stable and generic for $k_0$ large enough, as explained in Proposition \ref{p.fat}.\\
\begin{theorem}\label{t.main.2}
We use the same notations as in Theorem \ref{t.main.1}. Suppose that $n+d+mz=0$ and that either $n=1$ or $\sigma_A$ is a $k_0$-admissible symbol for some $k_0\geq 2$. For each $x,y\in U$ let
\[
Y_P(x,y)=\int_{S^*_y}\sigma_P\left(x,y,\partial_{x,y}(\partial_\xi\psi(x,y,0)\xi)\right)d_y\nu(\xi)\, .
\]
Then, there exists $V\subset U$ an open neighborhood of $0$ such that the following holds.
\begin{enumerate}
\item Uniformly for $(x,y)\in V\times V$ and $L\geq 1$,
\[
PK_L(x,y)=\frac{1}{(2\pi)^n}Y_P(x,y)\Big[\ln\left(L^{1/m}\right)-\ln_+\left(L^{1/m}|x-y|\right)\Big]+O(1)\, .
\]
\item There exists $Q\in L^\infty(V\times V)$ such that, uniformly for  $\kappa\geq 1$, $L\geq 1$ and $(x,y)\in V\times V$ such that $|x-y|\geq \kappa L^{-1/m}$,
\[
PK_L(x,y)=-\frac{1}{(2\pi)^n}Y_P(x,y)\ln(|x-y|)+Q(x,y)+O\left(\kappa^{-1/k_0}\right)\, .
\]
Here, if $n=1$ we set $k_0=1$.
\end{enumerate}

\end{theorem}
We prove Theorem \ref{t.main.2} in Section \ref{s.proofs.2}. As we will see below, the second point of Theorem \ref{t.diff.2} follows directly from this theorem.
\begin{remark}\label{r.examples}
The admissibility condition on the symbol of $A$ may appear to be unfamiliar. However, in practice, it is often satisfied. Here are two important examples of families of admissible symbols:
\begin{itemize}
\item If $n\geq 2$ and the level sets $S^*_x$ are strictly convex, $\partial^2_\xi\sigma_A$ is positive when restricted to their tangent spaces. Therefore, it cannot be a multiple of $\left(\partial_\xi\sigma\right)^{\otimes 2}$ so Theorem \ref{t.main.2} applies with $k_0=2$.
\item If $\sigma_A$ is a positive homogeneous polynomial of degree $m\in\N$ in $\xi$, $m\geq 1$, then it is $m$-admissible. Indeed, otherwise, taking $k=k_0=m$, we would have, for some $(x,\xi)\in U\times\left(\R^n\setminus\{0\}\right)$, $\sigma_A(x,\xi)^{m-1}\partial_\xi^m\sigma_A(x,\xi)=0$. But since $\sigma_A(x,\xi)>0$ we have $\partial_\xi^m\sigma_A(x,\xi)=0$ which implies that all the coefficients of $\sigma_A(x,\cdot)$ vanish. This contradicts the fact that $\sigma_A(x,\xi)>0$. In particular, Theorem \ref{t.main.2} applies for all differential operators.
\end{itemize}
\end{remark}
In addition to the two examples of the last remark, we prove the following theorem.
\begin{theorem}\label{t.fat}
Fix $n\in\N$, $n\geq 2$ and let $U\subset\R^n$ be an open subset. There exists $k_0=k_0(n)\in\N$ such that for each $m>0$, the set of $k_0$-admissible symbols is open and dense in the set of positive $m$-homogeneous symbols on $U$ for the topology induced by the Whitney topology through the restriction map to the (Euclidean) unit sphere bundle on $U$.
\end{theorem}
Theorem \ref{t.fat} follows immediately from Proposition \ref{p.fat}, which is proved Subsection \ref{ss.whitney}. The integer $k_0$ is explicit (see Proposition \ref{p.fat}).\\

Finally, though we do not use this in the proof of Theorems \ref{t.diff.1} and \ref{t.diff.2}, we prove the following result, which might be useful in further applications.
\begin{theorem}\label{t.main.3}
We use the same notations as in Theorem \ref{t.main.1}. Suppose that $n+d+mz_1<0$. Then, there exists and a function $K_\infty\in C^d(U\times U)$ such that the following holds. For each compact subset $\Omega\subset U\times U$, uniformly for $(x,y)\in \Omega$,
\[
PK_L(x,y)=PK_\infty(x,y)+O\left(L^{z_1+(n+d)/m}\right)\, .
\]
\end{theorem}
\begin{remark}
In Theorems \ref{t.main.1}, \ref{t.main.2}  and \ref{t.main.3}, the setting provided in Subection \ref{ss.setting} only comes into play through Theorem \ref{t.hormander}. Therefore, if one could weaken the hypotheses for this theorem, one would automatically extend Theorems \ref{t.main.1} and \ref{t.main.2} as a corollary. In particular, since H\"ormander proves Theorem \ref{t.hormander} for $P=Id$ without any compactness assumption or boundary condition, both of these results remain valid in this case.
\end{remark}
Let us check that Theorems \ref{t.main.1} and \ref{t.main.2} imply the results presented in the introduction.
\begin{proof}[ Proof of Theorem \ref{t.diff.1}]
Both results follow from Theorem \ref{t.main.1} applied to the first setting of Subsection \ref{ss.setting} with $z=-s$ by taking $P=\partial^\alpha_x\partial^\beta_y$ in a neighborhood of $0$. In this case, the order of $P$ is $d=|\alpha|+|\beta|$ and we have
\[
\sigma_P(x,y,\xi)=(i\xi)^\alpha(-i\xi)^\beta
\]
for any $\xi\in\R^n$ and $x,y\in V$ close enough to $0$.
\end{proof}
\begin{proof}[ Proof of Theorem \ref{t.diff.2}]
Set $z=-s=-n/m$. For the first part, set $P=\partial^\alpha_x\partial^\beta_y$ near $0$ and proceed as in the proof of Theorem \ref{t.diff.1}. Indeed, since $(\alpha,\beta)\neq (0,0)$, we have $n+d+mz_1=|\alpha|+|\beta|>0$. For the second part, since by Remark \ref{r.examples} $\sigma_A$ is $m$-admissible, and since $n+d+mz_1=0$, we apply Theorem \ref{t.main.2} instead. In our case, $P=Id$ so for each $x,y\in U$, $Y_P(x,y)=\nu_y\left(S^*_y\right)$ so
\[
K^s_L(x,y)=\frac{1}{(2\pi)^n}\nu_y\left(S^*_y\right)\left[\ln\left(L^{1/m}\right)-\ln_+\left(L^{1/m}|x-y|\right)\right]+O(1)\, .
\]
But since $K^s_L(x,y)=\overline{K^s_L(y,x)}$, we may replace $\nu_y\left(S^*_y\right)$ in the above expression by $\frac{\nu_x\left(S^*_x\right)+\nu_y\left(S^*_y\right)}{2}$.
as announced.
\end{proof}

\begin{proof}[Proof of Corollary \ref{co.diff.1}]
We use the notations of Theorem \ref{t.main.2}. First of all, by definition, as $L\rightarrow +\infty$, $K_L^s\rightarrow K$ in distribution. Moreover, by Theorem \ref{t.diff.2}, any point in $\man$ has a neighborhood $V$ such that the sequence $(K_L^s)_{L\geq 1}$ is uniformly bounded on $V\times V$ by a locally integrable function and converge pointwise towards $-g_A(x,y)\ln\left(|x-y|\right)+Q(x,y)$ where $Q\in L^\infty(V\times V)$ on the complement of the diagonal in $V\times V$. In particular, they converge in distribution to this function. This implies that when restricted to $C^\infty(V\times V)$,
\[
K(x,y)=-g_A(x,y)\ln\left(|x-y|\right)+Q(x,y)\, .
\]
Now, given any smooth distance $d$ on $\man$, for each $x,y$ distinct,
\[
\ln\left(|x-y|\right)=\ln\left(d(x,y)\right)+\ln\left(\frac{|x-y|}{d(x,y)}\right)
\]
and the second term is bounded so, on $V\times V$,
\[
K(x,y)=-g_A(x,y)\ln\left(d(x,y)\right)+Q_A(x,y)
\]
for some $Q_A\in L^\infty(V\times V)$. But $K$ is the integral kernel of $A^s$ which is a self-adjoint pseudo-differential operator (see \cite{see67} or Proposition 29.1.9 of \cite{ho_apdo4}). In particular, it is smooth and symmetric outside the diagonal (see for instance Theorem 18.1.16 of \cite{ho_apdo3}). Hence, $Q_A$ must also be symmetric and smooth outside the diagonal.
\end{proof}

\begin{remark}
These proofs work also in Setting 2 of Subsection \ref{ss.setting} instead of that of Setting 1, so all of the results from the introduction hold also in Setting 2.
\end{remark}

\section{Heuristics and proof outline}\label{s.heuristics_outline}

In this section we provide a heuristic justification for Theorems \ref{t.main.1} and \ref{t.main.2} and an outline of the skeleton of the proof. At the end of this section, we also provide a proof map to highlight the dependencies between intermediate results leading to the proofs of Theorems \ref{t.main.1}, \ref{t.main.2} and \ref{t.main.3}, see Figure 1.\\

\subsection{Heuristics}\label{ss.heuristics}

In order to get a sense of the kind of calculations we will carry out in the rest of the article, let us present a simple example, with few non-rigorous steps in order to shorten the argument. We assume that $\man$ is a closed Riemmanian manifold and that $A$ denotes the associated Laplacian\footnote{Here we use the convention that $\Delta = -\textup{div}\circ\nabla$ so that the operator is positive.}. Then, $A$ is indeed elliptic of order $m=2$ and self-adjoint with respect to the riemmanian volume density $d\mu$. Moreover the symbol of $A$ is $\sigma_A(x,\xi) = \|\xi\|_x^2$ where $\|\cdot\|_x$ is the norm induced by the riemmanian metric on $T^*_x\man$. Thus, in orthonormal coordinates $S^*_x = S^{n-1}$. Finally, we take $P=Id$. Now, if $s\geq \frac{n}{2}$,
\begin{equation}\label{e.heuristic.1}
K^s_L(x,y)=\int_1^L \lambda^{-s} E_\lambda '(x,y)d\lambda + O(1)=s\int_1^L\lambda^{-s-1} E_\lambda(x,y)d\lambda+O(1)\, .
\end{equation}
Here, we artificially cut-off the first eigenvalues since they contribute a constant term to the sum defining $K^s_L$. By Theorem \ref{t.hormander},
\begin{equation}\label{e.heuristic.2}
E_L(x,y) \simeq \frac{1}{(2\pi)^n}\int_{|\xi|^2\leq L} e^{i\psi(x,y,\xi)} d\xi
\end{equation}
where
\begin{equation}\label{e.heuristic.3}
\psi(x,y,\xi)\simeq \langle x-y,\xi\rangle\, .
\end{equation}
Replacing $E_\lambda'$ in Equation \eqref{e.heuristic.1} the expression given by Equations \eqref{e.heuristic.2} and \eqref{e.heuristic.3} we get
\[
K^s_L(x,y)\simeq \frac{s}{(2\pi)^{n/2}}\int_1^L \lambda^{-s-1}\int_{\|\xi\|^2\leq \lambda}e^{i\langle x-y,\xi\rangle} d\xi d\lambda\, .
\]
At this point, we make the additional assumption that $n\geq 2$. The one dimensional case is similar in spirit but follows a different argument. We then continue with a polar change of coordinates: $\xi=t\omega$.
\begin{equation}\label{e.heuristic.4}
K^s_L(x,y)\simeq \frac{s}{(2\pi)^n}\int_1^L\lambda^{-s-1} \int_0^{\lambda^{1/2}}J_{x,y}(t) t^{n-1}dt d\lambda
\end{equation}
where
\[
J_{x,y}(t)=\int_{S^{n-1}} e^{it\langle x-y,\omega\rangle} d\omega\, .
\]
Let us first assume that $s>\frac{n}{2}$ and take $x-y = L^{-1/2}h$ where $h\in\R^n$ is fixed. Then,
\begin{align*}
K^s_L(x,y)&\simeq \frac{1}{(2\pi)^n}\int_{S^{n-1}} \int_1^L  \lambda^{-s} \int_0^{\lambda^{1/2}}  e^{i L^{-1/2}t \langle h,\omega \rangle}t^{n-1} dt d\lambda d\omega\\
&\simeq \frac{s}{(2\pi)^n}L^{n/2-s}\int_{S^{n-1}}\int_0^1 \lambda^{-s-1} \int_0^{\lambda^{1/2}} e^{it \langle h,\omega \rangle}t^{n-1}dt d\lambda d\omega\\
&\simeq\frac{1}{(2\pi)^n} \int_{\|\xi\|^2\leq 1} \|\xi\|^{-s} e^{i\langle h,\xi\rangle} d\xi L^{n/2-s}\, .
\end{align*}
This is the conclusion of Theorem \ref{t.diff.1}.\\

Assume now that $s=\frac{n}{2}$. Starting off from Equation \eqref{e.heuristic.4}, we get
\[
K_L^s(x,y)\simeq \frac{s}{(2\pi)^n}\int_{|x-y|^2}^{|x-y|^2L}\lambda^{-s-1}\int_0^{\lambda^{1/2}}J_{x,y}(|x-y|^{-1}t) t^{n-1}dt d\lambda\, .
\]
Note that, by the stationary phase method,
\begin{equation}\label{e.heuristic.5}
J_{x,y}(t) = O\left((|x-y|t)^{-1/2}\right)\, .
\end{equation}
This crucial observation basically allows us to replace $J_{x,y}(|x-y|^{-1}t)$ with $|S^{n-1}|\one[|x-y|^{-1}t\leq 1]$ and get
\begin{align*}
K_L^s(x,y)&\simeq \frac{s|S^{n-1}|}{(2\pi)^n}\int_{|x-y|^2}^{|x-y|^2L}\lambda^{-s-1}\int_0^{\lambda^{1/2}}\one[|x-y|^{-1}t\leq 1] t^{n-1}dtd\lambda\\
& = \dots\\
& = \frac{|S^{n-1}|}{(2\pi)^n}\int_{|x-y|}^{|x-y|L^{1/2}}\one[u\leq 1]\frac{du}{u}+O(1)\\
& = \frac{|S^{n-1}|}{(2\pi)^n}\ln\left(L^{1/2}\right)-\ln_+\left(L^{1/2}|x-y|\right) + O(1)\, .
\end{align*}
This is the essential statement of Theorem \ref{t.diff.2}.

\subsection{Proof strategy}\label{ss.outline}

There are two main obstacles to carry out the above calculation rigorously in the general case and Sections \ref{s.preliminaries} and \ref{s.singularities} are devoted to dealing with them. The first is to justify Equation \eqref{e.heuristic.3}. This is the role of Lemmas \ref{l.toolbox}, \ref{l.dimension_1} and \ref{l.jatz} that roughly state that $\psi$ behaves like $\langle x-y,\xi\rangle$. The second difficulty is to obtain an analog of Equation \eqref{e.heuristic.5} when $S^{n-1}$ is replaced by $S^*_x=\{\xi\, ,\, \sigma_A(x,\xi)=1\}$ for a general symbol $\sigma_A$. Indeed, in this case, the standard stationary method need not apply and we must use more general results on oscillatory integrals. This requires the assumption that $\sigma_A$ be admissible (see Definition \ref{d.admissible}). To make this point more precise, let us introduce some notation.\\

As in the previous section, we fix once and for all a point in $\man$ and consider a local chart centered at this point defined on $U\subset\R^n$ given by Theorem \ref{t.hormander}. We also take $P$ with principal symbol $\sigma_P$, $W\subset U\times U$ and $\psi\in C^\infty(U\times U\times\R^n)$ as in this theorem. The following quantity will be central in our proofs. For any $t>0$, $x,y\in U$ and $\xi\in\R^n$ let
\begin{equation}\label{e.amplitude}
H_P(x,y,\xi,t)=e^{i\psi(x,y,t\xi)}\sigma_P\left(x,y,t^{-1}\partial_{x,y}\psi(x,y,t\xi)\right)
\end{equation}
and
\begin{equation}\label{e.osc}
  J_A(x,y,t)=\int_{S^*_y}H_P(x,y,\xi,t)d_y\nu(\xi).
\end{equation}
Since $\sigma_P$ is $d$-homogeneous in its third variable, $H_P$ satisfies the following Equation. For any $s,t>0$, $x,y\in U$ and $\xi\in\R^n$,
\begin{equation}\label{e.scaling}
H_P(x,y,s\xi,t)=s^dH_P(x,y,\xi,st).
\end{equation}
We will prove the following proposition.

\begin{proposition}\label{p.jprop}
Assume that $\sigma_A$ is $k_0$-admissible. Then, there exists $V\subset U$ an open neighborhood of $0$ and $C<+\infty$ such that, uniformly for distinct $x,y\in V$ and $t>0$
\[
|J_A(x,y,t)|\leq C\left(t|x-y|\right)^{-\frac{1}{k_0}}\, .
\]
\end{proposition}

The proof of Proposition \ref{p.jprop} is divided into two steps. First, we will prove that the admissibility condition on $\sigma_A$ implies a property governing the decay of certain oscillatory integrals over the level sets of $\sigma_A$ that we define below in Definition \ref{d.ndls}. Next, we prove that this property implies the required behavior of $J_A$. More precisely, we introduce the following terminology.

\begin{definition}\label{d.ndls}
Let $\eps>0$, $m>0$, $E\subset\R^p$ a neighborhood of $0$ and let $U\subset\R^n$ be an open subset. Let $\sigma\in C^\infty(U\times\R^n\zeroless)$ be homogeneous of degree $m$ in the second variable. For each $x\in U$ let $S^*_x =\{\xi\in\R^n\ |\ \sigma(x,\xi)=1\}$ and $d_x\mu$ be the area measure on $S^*_x$. Let $S^*U=\{(x,\xi)\in U\times \R^n\ |\ \xi\in S^*_x\}$.
\begin{enumerate}
\item    Given a compact subset $\Omega\subset U\times(\R^n\zeroless)$ let $X=\{(x,\tau,\xi)\in \Omega\times\R^n\ |\ \xi\in S^*_x\}$. We call a \textbf{deformation of the height function for $\sigma$ over $\Omega$} any family $(f_\eta)_{\eta\in E}$ of continuous, real-valued functions on $X$, smooth in the third variable $\xi$, with the following properties:
\begin{itemize}
\item for each $(x,\tau,\xi)\in \Omega\times \R^n$ such that $\xi\in S^*_x$, $f_0(x,\tau,\xi)=\langle\tau,\xi\rangle$
\item for each $\alpha\in\N^n$, the map $\eta \mapsto \partial^\alpha_\xi f_\eta$ is continuous for the topology of uniform convergence on compact sets.
\end{itemize}
\item We say that $\sigma$ has \textbf{$\eps$-non-degenerate level sets} if, for any compact subset $\Omega$ of $U\times (\R^n\zeroless)$ and any deformation of the height function $(f_\eta)_\eta$ for $\sigma$ over $\Omega$ there exists $V\subset\R^p$ a neighborhood of $0$ depending only on $\Omega$ (and $\eps$) such that for each $\gamma\in C^\infty(S^*U)$ and each continuous family of smooth functions on $(u_\eta)_{\eta}\in \left(C^\infty\left(\R^n\right)\right)^E$, there exists $C<+\infty$ such that for each each $\eta\in V$, each $(x,\tau)\in \Omega$ and each $\lambda>0$,
\begin{equation}\label{e.decay}
\Big|\int_{S^*_x}e^{i\lambda f_\eta(x,\tau,\xi)}u_\eta(\xi)\gamma(x,\xi)d_x\mu(\xi)\Big|\leq C\lambda^{-\eps}\, .
\end{equation}
We say that $\sigma$ has \textbf{non-degenerate level sets} if it has $\eps$-non-degenerate level sets for some $\eps>0$.
\item Let $\eps>0$. We say that a homogeneous symbol on a manifold has non-degenerate (resp. $\eps$-non-degenerate) level sets if it has this property when written in any local coordinate system.
  \end{enumerate}
\end{definition}
Proposition \ref{p.jprop} will then be a consequence of the following results. On the one hand, we will prove:
\begin{proposition}\label{p.color.1}
Fix $n,k_0\in\N$, $n\geq 1$, $k_0\geq 2$. Let $U\subset\R^n$ be an open subset and let $\sigma\in C^\infty(U\times(\R^n\setminus\{0\})$ be positive and homogeneous of degree $m>0$ in its second variable. If $\sigma$ is $k_0$-admissible, then it has $\frac{1}{k_0}$-non-degenerate level sets.
\end{proposition}
The proof of this proposition, which is presented in Subsection \ref{ss.color}, is entirely independent of the rest of the present text and uses different techniques. It is followed by Subsection \ref{ss.whitney}, in which we prove that the admissibility condition is generic in a suitable sense.\\

On the other hand, in Subsection \ref{ss.jp}, we will prove the following result.
\begin{lemma}\label{l.jati}
Fix $\eps>0$. Suppose that the symbol $\sigma_A$ has $\eps$-non-degenerate level sets (see Definition \ref{d.ndls}). Then, there exists $V\subset U$ an open neighborhood of $0$ and $C<+\infty$ such that, uniformly for distinct $x,y\in V$ and $t>0$
\[
|J_A(x,y,t)|\leq C\left(t|x-y|\right)^{-\eps}\, .
\]
\end{lemma}

This corresponds to Proposition 23 of \cite{ale161} for $\eps=\frac{1}{2}$ although, in that setting, the non-degeneracy condition was always satisfied.\\

In the one dimensional case, Proposition \ref{p.jprop} is replaced by Lemma \ref{l.decay_one}.\\

After proving all of these results, we carry out the calculation sketched in Subsection \ref{ss.heuristics} in Sections \ref{s.proofs.1} and \ref{s.proofs.2} as we will now explain in more detail. We therefore suggest that the reader have Subsection \ref{ss.heuristics} in mind for what follows. The integration by parts is of course valid in a general setting. This allows us to obtain an expression like Equation \eqref{e.heuristic.1} where the map $\lambda^{-s}$ is replaced by $f(\lambda)$ for some adequate function $f$. More explicitely, in Section \ref{s.proofs.1}, we derive the following result. We start by introducing a suitable function $f:]0,+\infty[\mapsto\C$ and studying the asymptotics of the following kernel:
\[
K_L^f:(x,y)\mapsto\sum_{\lambda_k\leq L}f(\lambda_k)e_k(x)\overline{e_k}(y)\, .
\]
This is again a smooth function. Since all of our results are local, we fix once and for all a point in $\man$ and consider $x=(x_1,\dots,x_n)$ the local coordinate system at this point provided by Theorem \ref{t.hormander}, defined on an open neighborhood $U$ of $0$ in $\R^n$.
\begin{proposition}\label{p.main}
Take $f:\R\rightarrow \C$ with support in $]0,+\infty[$ differentiable almost everywhere. Then, in local coordinates, uniformly for each  $x,y\in U$, for each $L>0$,
\begin{multline*}
PK_L^f(x,y)=\frac{1}{(2\pi)^n}\int_{\sigma_A(y,\xi)\leq L}e^{i\psi(x,y,\xi)}\sigma_P(x,y,\partial_{x,y}\psi(x,y,\xi))f(\sigma_A(y,\xi))d\xi\\
+O\left(f(L)L^{(n+d-1)/m}\right)+O\left(\int_0^Lf'(\lambda)\lambda^{(n+d-1)/m}d\lambda\right).
\end{multline*}
In addition, uniformly for any $(x,y)\in \left(U\times U\right)\setminus W$, for each $L\geq 1$,
\[
PK_L^f(x,y)=O\left(f(L)L^{(n+d-1)/m}\right)+O\left(\int_0^Lf'(\lambda)\lambda^{(n+d-1)/m}d\lambda\right).
\]
Finally, the constants implied by the $O$'s do not depend on $f$. 
\end{proposition}
The condition on the support of $f$ can be achieved by multiplying $f$ by a suitable cut-off function when necessary since the spectrum of $A$ is bounded from below. We prove Proposition \ref{p.main} in Section \ref{s.proofs.1}. Then, we consider the case where $f$ is of the form $f(t)=\chi(t)t^z$ where $z=z_1+iz_2\in\C$ and $\chi$ is some smooth function with support in $]0,\pfty[$ equal to $1$ for $t$ large enough. In Section \ref{s.proofs.1}, , we prove Theorem \ref{t.main.3} using only a crude estimate from Theorem \ref{t.hormander}, and  we also deduce Theorem \ref{t.main.1} from Proposition \ref{p.main} and results from Section \ref{s.preliminaries}. Next, in Section \ref{s.proofs.2} we prove Theorem \ref{t.main.2} using again Proposition \ref{p.main} but also Proposition \ref{p.jprop}. We end this section with a diagram detailing the dependencies between different results involved in the proofs of Theorems \ref{t.main.1}, \ref{t.main.2} and \ref{t.main.3}.

\vfill
\pagebreak

\begin{center}
\begin{tikzpicture}[node distance = 1cm, auto,]
\node[punkt] (t_subcrit) {Theorem \ref{t.main.3}};
\node[punkt, above= of t_subcrit] (t_horm) {Theorem \ref{t.hormander}}
 edge[pil] (t_subcrit.north);
\node[punkt, left=of t_horm] (l_51) {Lemma \ref{l.link}}
 edge[pil, bend right=10] (t_subcrit.west);
\node[punkt,right=4cm of t_subcrit] (t_supercrit) {Theorem \ref{t.main.1}};
\node[punkt, above= of t_supercrit] (p_main) {Proposition \ref{p.main}}
 edge[pil] (t_supercrit.north);
\node[punkt, above=of p_main] (l_51) {Lemma \ref{l.link}}
 edge[pil] (p_main.north);
\node[punkt, left=of l_51] (t_horm) {Theorem \ref{t.hormander}}
 edge[pil, bend right=20] (t_supercrit.west)
 edge[pil, bend right=10] (p_main.west);
\node[punkt, right=of l_51] (l_41) {Lemma \ref{l.toolbox}}
 edge[pil, bend left=20] (t_supercrit.east);
\end{tikzpicture}
\vspace{2cm}$ $\\
\begin{tikzpicture}[node distance = 1cm, auto,]
\node[punkt] (t_crit) {Theorem \ref{t.main.2}};
\node[punkt, above=of t_crit] (l_62) {Lemma \ref{l.control II}}
 edge[pil] (t_crit.north);
\node[punkt, left= of l_62] (p_main) {Proposition \ref{p.main}}
 edge[pil, bend right=10] (t_crit.west);
\node[punkt, right=of l_62] (l_61) {Lemma \ref{l.control I}}
 edge[pil, bend left=10] (t_crit.east);
\node[punkt, above= of l_62] (l_43) {Lemma \ref{l.jatz}}
 edge[pil] (l_62.north);
\node[punkt, right= of l_43] (l_34) {Lemma \ref{l.jati}}
 edge[pil] (l_61.north);
\node[punkt, right= of l_34] (l_44) {Lemma \ref{l.decay_one}}
 edge[pil, bend left=10] (l_61.east);
\node[punkt, above= of l_43] (l_41) {Lemma \ref{l.toolbox}}
 edge[pil] (l_43.north)
 edge[pil, bend left=10] (l_34.north)
 edge[pil, bend left=10] (l_44.north);
\node[punkt, above=of p_main] (t_horm) {Theorem \ref{t.hormander}}
 edge[pil] (p_main.north);
 \node[punkt, left= of t_horm] (l_51) {Lemma \ref{l.link}}
 edge[pil, bend right=10cm] (p_main.west);
\node[punkt, right=of l_44] (p_33) {Proposition \ref{p.color.1}}
 edge[pil, bend left=15] (l_61.east);
\node[punkt, right=4cm of l_41] (l_42) {Lemma \ref{l.dimension_1}}
 edge[pil] (l_44.north);
\node[punkt, right= of l_42] (l_71) {Lemma \ref{l.color}}
 edge[pil] (p_33.north);
\end{tikzpicture}
\vspace{1cm}$ $\\
A map of the proofs of Theorems \ref{t.main.1}, \ref{t.main.2} and \ref{t.main.3}. The result at the origin of each arrow is used in the proof of the result at its target.
\end{center}

\pagebreak

\section{Preliminary results}\label{s.preliminaries}

As before, in this section we fix once and for all a point in $\man$ and consider a local chart centered at this point defined on $U\subset\R^n$ given by Theorem \ref{t.hormander}. We also take $P$ with principal symbol $\sigma_P$, $W\subset U\times U$ and $\psi\in C^\infty(U\times U\times\R^n)$ as in this theorem. The object of this section is to estimate the behavior of the phase $\psi$ near the diagonal and to prove Lemma \ref{l.jatz}.

\subsection{Basic properties of the phase $\psi$}\label{ss.psi}

The phase $\psi$ from Theorem \ref{t.hormander} will frequently appear in the calculations below. We begin by deducing a list of properties of $\psi$ from those given in Definition \ref{d.phase}. We gather these properties in Lemma \ref{l.toolbox}. It is easy to check that all these properties are satisfied by the function $\psi(x,y,\xi)=\langle x-y,\xi\rangle$. Next, we present an additional lemma, Lemma \ref{l.dimension_1}, for the case $n=1$. Finally, we use Lemma \ref{l.toolbox} to deduce some properties of the function $H_P$ defined in Equation \eqref{e.amplitude}.

\begin{lemma}\label{l.toolbox}
  Let $U\subset\R^n$ and let $\psi\in C^\infty(U\times U\times\R^n)$ be a proper phase function. For each $t>0$, let $\psi_t=t^{-1}\psi(\cdot,\cdot,t\cdot)$. Then,
  \begin{enumerate}
  \item For each $x,y\in U$ and each $t>0$, $\psi_t(x,y,0)=0$.
  \item For each $x\in U$, each $t>0$ and each $\xi\in\R^n$, $\psi_t(x,x,\xi)=0$.
  \item For each $x\in U$, each $t>0$ and each $\xi\in\R^n$, $\partial_{x,y}\psi_t(x,x,\xi)=(\xi,-\xi)$.
  \item The sequence $(\psi_t)_{t>0}$ converges in $C^\infty(U\times U\times\R^n)$ as $t\rightarrow 0$ to the function $\psi_0$ defined by $\psi_0(x,y,\xi)=\partial_\xi\psi(x,y,0)\xi$. In other words, for each compact subset $\Omega\subset U$, each $R<+\infty$ and each $\alpha,\beta,\gamma\in\R^n$,
    \[\lim_{t\rightarrow 0}\sup_{x,y\in\Omega,\ |\xi|\leq R}\left|\partial^\alpha_\xi\partial^\beta_x\partial^\gamma_y\psi_t(x,y,\xi)-\partial^\alpha_\xi\partial^\beta_x\partial^\gamma_y\psi_0(x,y,\xi)\right|=0.\]
   \item The sequence $(\psi_t)_{t\geq 0}$ is bounded in $C^\infty(U\times U\times\R^n)$.
  \end{enumerate}
\end{lemma}
\begin{prf}[ of Lemma \ref{l.toolbox}]
  Let $t>0$, $x,y\in U$ and $\xi\in\R^n$. Then, $\langle x-y,0\rangle = \langle x-x,t\xi\rangle =0$ so $\psi_t(x,x,\xi)=\psi_t(x,y,0)=0$ by the second point of Definition \ref{d.phase}. This proves the first two points of Lemma \ref{l.toolbox}. By point 3 of Definition \ref{d.phase}, for each $x\in U$ and $\xi\in\R^n$, $\partial_x\psi(x,x,\xi)=\xi$, so $\partial_x\psi_t(x,x,\xi)=t^{-1}(t\xi)=\xi$. Next, by differentiating the following equality
  \[\psi_t(x+sv,x+sv,\xi)=0\]
  with respect to $s\in\R$, at $s=0$, where $x\in U$, $\xi\in\R^n$ and $v\in\R^n$, we get
  \[\partial_x\psi_t(x,x,\xi)+\partial_y\psi_t(x,x,\xi)=0.\]
  This proves the third point of Lemma \ref{l.toolbox}.\\
To prove the fourth point, first, fix $\beta,\gamma\in\N^n$ and let $\Omega\subset U$ be a compact subset and $R<+\infty$. Then, for each $x,y\in\Omega$ and $\xi\in\R^n$ such that $|\xi|\leq R$,
  \[\partial^\beta_x\partial^\gamma_y\psi_t(x,y,\xi)=t^{-1}\partial^\beta_x\partial^\gamma_y\psi(x,y,t\xi).\]
  By the first point, of Lemma \ref{l.toolbox}, $\partial^\beta_x\partial^\gamma_y\psi(x,y,0)=0$. We apply Taylor's formula to $t\mapsto \partial^\beta_x\partial^\gamma_y\psi(x,y,t\xi)$ uniformly for $t\leq 1$, $x,y\in \Omega$ and $\xi\in\R^n$ such that $|\xi|\leq R$ and get
  \[t^{-1}\partial^\beta_x\partial^\gamma_y\psi(x,y,t\xi)=0+\partial^\beta_x\partial^\gamma_y(\partial_\xi\psi(x,y,0)\xi)+O(t).\,
  \]
  In particular, as $t\rightarrow 0$, $\partial^\beta_x\partial^\gamma_y\psi_t\rightarrow\partial^\beta_x\partial^\gamma_y\psi_0$ uniformly for $x,y\in \Omega$ and $\xi\in \R^n$, $|\xi|\leq R$. Next, fix $\alpha\in\N$ and suppose $|\alpha|\geq 1$. Then, for each $x,y\in K$, $\xi\in \R^n$, $|\xi|\leq R$ and $t>0$,
  \[\partial^\alpha_\xi\partial^\beta_x\partial^\gamma_y\psi_t(x,y,\xi)=t^{|\alpha|-1}\partial^\alpha_\xi\partial^\beta_x\partial^\gamma_y\psi(x,y,t\xi).\]
  If $|\alpha|=1$, as $t\rightarrow 0$ the right hand side converges uniformly to $\partial^\alpha_\xi\partial^\beta_x\partial^\gamma_y\psi(x,y,0)=\partial^\alpha_\xi\partial^\beta_x\partial^\gamma_y(\partial_\xi\psi(x,y,0)\xi)$. On the other hand, if $|\alpha|>1$, as $t\rightarrow 0$ it converges uniformly to $0=\partial^\alpha_\xi\partial^\beta_x\partial^\gamma_y(\partial_\xi\psi(x,y,0)\xi)$. This proves the fourth point of Lemma \ref{l.toolbox}. Lastly, the family $(\psi_t)_{t>0}$ is obviously continuous into $C^\infty(U\times U\times\R^n)$ for $t>0$. By the fourth point of Lemma \ref{l.toolbox} we may extend it by continuity to $t=0$. On the other hand, by the fifth point of Definition \ref{d.phase}, it also converges as $t\rightarrow\infty$. In particular, the family $(\psi_t)_{t\geq 0}$ is uniformly bounded in $C^\infty(U\times U\times\R^n)$. This proves the fifth point of Lemma \ref{l.toolbox}.
\end{prf}

We use the following lemma to prove Lemma \ref{l.decay_one} below, which is the analog of Proposition \ref{p.jprop} we use in dimension $n=1$. It is the only place where we use the fact that $\psi$ satisfies the eikonal equation \eqref{e.eikonal}.

\begin{lemma}\label{l.dimension_1}
Assume that $n=1$. For each segment $I\subset U$ there exists $c\in]0,+\infty[$ such that for each $x,y\in I$ and $\xi\in\R$, $\frac{1}{c}|x-y|\leq |\partial_\xi\psi(x,y,\xi)|\leq c|x-y|$ and  $|\partial_\xi^2\psi(x,y,\xi)|\leq c|x-y|(1+|\xi|)^{-1}$.
\end{lemma}

\begin{prf}[ of Lemma \ref{l.dimension_1}]
Let us fix $I\subset U$ a compact interval. Since the symbol $\sigma_A$ is $m$-homogeneous and $dim(\man)=1$ there exists a positive function $\varrho\in C^\infty(U)$ such that $\sigma_A(x,\xi)=\varrho(x)^m|\xi|^m$ for $\xi\neq 0$ and $x\in U$. By construction of $\psi$ there exist $C_1<+\infty$ and symbols $\tau\in S^0(U\times\R)$ and $\sigma\in S^1(U\times\R)$ such that $\sigma(x,\xi)=\varrho(x)|\xi|+\tau(x,\xi)$ for $|\xi|\geq C_1$ and $x\in U$ and such that
  \[\forall\xi\in\R\setminus[-C_1,C_1],\ \forall x,y\in U,\ \sigma(x,\partial_x\psi(x,y,\xi))=\sigma(y,\xi).\]
  Since $\tau\in S^0$ and since $\varrho$, being positive and continuous, is bounded from below on $I$, there exists  $C_2\in [\max(C_1,1),+\infty[$ such that for any $x\in I$ and $\xi\in\R$ such that $|\xi|\geq C_2$,
  \begin{align*}
    \frac{1}{2}\varrho(x)|\xi|&\leq\sigma(x,\xi)\leq 2\varrho(x)|\xi|\, ;\\
    C_2^{-1}&\leq \textup{sign}(\xi)\partial_\xi\sigma(x,\xi)\leq C_2.
  \end{align*}
  Let $(\sigma^{-1})(x,\cdot)$ be the inverse of $\sigma(x,\cdot):[C_2,+\infty[\rightarrow[\sigma(x,C_2)+\infty[$. Let us fix $x_0\in I$. Then, for any $x\in I$,
  \begin{equation}\label{e.dim.1}
  \partial_x\psi(x,x_0,\xi)=(\sigma^{-1})(x,\sigma(x_0,\xi)).\,
  \end{equation}
  Differentiating this Equation with respect to $\xi$ we obtain the following expression for $\partial_\xi\partial_x\psi$.
   \[
   \partial_\xi\partial_x\psi(x,x_0,\xi)=\partial_\xi(\sigma^{-1})(x,\sigma(x_0,\xi))\partial_\xi\sigma(x_0,\xi)\, .
   \]
  Now, by definition of $\sigma^{-1}$, we have, for $x\in I$ and $\xi\in\R$ such that $\xi\geq C_3=\max_{y\in I}\sigma(y,C_2)$,
  \[
  \partial_\xi(\sigma^{-1})(x,\xi)=\left(\partial_\xi\sigma(x,\sigma^{-1}(x,\xi))\right)^{-1}=\left(\varrho(x)+\partial_\xi\tau(x,\sigma^{-1}(x,\xi))\right)^{-1}\, ,
  \]
  where $\varrho(x)$ is bounded on $I$ from above and below by positive constants and $\partial_\xi\tau(x,\sigma^{-1}(x,\xi))$ is $O(|\sigma^{-1}(x,\xi)|^{-1})$ uniformly for $x\in I$. Since $\sigma^{-1}(x,\xi)\xrightarrow[\xi\to\pfty]{}\pfty$ then there exists $C_4>0$ such that for any $x\in I$ and any $\xi\geq C_4\geq \max(C_3,C_2)$,
  \begin{equation}\label{e.dim.2}
  C_4^{-1}\leq\partial_\xi(\sigma^{-1})(x,\xi)\leq C_4.\,
  \end{equation}
  Therefore,
  \[
  C_2^{-1}C_4^{-1}\leq\partial_x\partial_\xi\psi(x,x_0,\xi)\leq C_2C_4\,.\]
  Recall that, by the first point of Lemma \ref{l.toolbox}, $\psi(x,x,\xi)=0$ for any $x\in U$ and any $\xi\in\R$. Thus, for any $x\in I$, $\xi\geq C_4$,
  \[
  |\partial_\xi\psi(x,x_0,\xi)|=\left|\int_{x_0}^x\partial_\xi\partial_x\psi(y,x_0,\xi)dy\right|\in\left[C_5^{-1}|x-x_0|,C_5|x-x_0|\right]
  \]
where $C_5=C_2C_4$ is independent of the choice of $x_0$. The case where $\xi<0$ is symmetric and this proves the first identity announced in the lemma. For the second identity, we start by differentiating Equation \eqref{e.dim.1} with respect to $\xi$ to obtain
\begin{equation}\label{e.dim.3}
\partial_\xi^2\partial_x\psi(x,x_0,\xi)=\partial_\xi^2(\sigma^{-1})(x,\sigma(x_0,\xi))(\partial_\xi\sigma(x_0,\xi))^2+\partial_\xi(\sigma^{-1})(x,\sigma(x_0,\xi))\sigma_\xi^2(x_0,\xi)\, .
\end{equation}
To deal with the second term of the right hand side, observe that, since $\sigma$ is a symbol of order one and by Equation \eqref{e.dim.2}, there exists a constant $C_6<+\infty$ such that for any $x,x_0\in I$ and any $\xi\in\R$,
\begin{equation}\label{e.dim.4}
\left|\partial_\xi(\sigma^{-1})(x,\sigma(x_0,\xi))\sigma_\xi^2(x_0,\xi)\right|\leq C_6(1+|\xi|)^{-1}.\, 
\end{equation}
For the first term we proceed as follows. By definition of $\sigma^{-1}$, we have, for any $x\in I$ and $\xi\geq C_3$,
\[
\partial_\xi^2\sigma(x,\sigma^{-1}(x,\xi))(\partial_\xi(\sigma^{-1})(x,\xi))^2+\partial_\xi\sigma(x,\sigma^{-1}(x,\xi))\partial_\xi^2(\sigma^{-1})(x,\xi)=0.\, 
\]
By Equation \eqref{e.dim.2}, since $\sigma$ is a symbol of order one and since $\partial_\xi\sigma$ is bounded from below on $[C_2,+\infty[$, there exists $C_7<+\infty$ such that for each $x,x_0\in I$ and $\xi\geq C_3$,
\begin{equation}\label{e.dim.5}
\left|\partial_\xi^2(\sigma^{-1})(x,\sigma(x_0,\xi))(\partial_\xi\sigma(x_0,\xi))^2\right|\leq C_7(1+|\xi|)^{-1}\, .
\end{equation}
We use Equations \eqref{e.dim.4} and \eqref{e.dim.5} on the right hand side of Equation \eqref{e.dim.3} and get, for each $x,x_0\in I$ and $\xi\geq C_3$,
\[
\left|\partial_x\partial_\xi^2\psi(x,x_0,\xi)\right|\leq (C_6+C_7)(1+|\xi|)^{-1}\, .
\]
As before, since for all $x\in I$ and $\xi\in\R$, $\psi(x,x,\xi)=0$, we have
\[
\left|\partial_\xi^2\psi(x,x_0,\xi)\right|\leq\int_{x_0}^x\left|\partial_x\partial_\xi^2\psi(y,x_0,\xi)\right|dy\leq C_8|x-x_0|(1+|\xi|)^{-1}
\]
where $C_8=C_6+C_7$. The case $\xi<0$ is symmetric.
\end{prf}

From Lemma \ref{l.toolbox}, we deduce the following properties of the function $H_P$ defined in Equation \eqref{e.amplitude}.
\begin{lemma}\label{l.jatz}
The function $H_P$ satisfies the following properties.
\begin{enumerate}
\item The function $t\mapsto H_P(\cdot,\cdot,\cdot,t)$ extends continuously to $t=0$ as a function from $\R_+$ to $C^\infty(U\times U\times\R^n)$ and
\[H_P(x,y,\xi,0)=\sigma_P\left(x,y,\partial_{x,y}(\partial_\xi\psi(x,y,0)\xi)\right).\]
\item Uniformly for $t\geq 0$ and $x,y$ in compact subsets of $U$ and $\xi\in \R^n$,
\[
H_P(x,y,\xi,t)-H_P(x,y,\xi,0)=O\left(t|x-y||\xi|^{d+1}\right)\, .
\]
\end{enumerate}
\end{lemma}
Note that the assertions are both easy to check for the prototype $H_P(x,y,t)=e^{it\langle x-y,\xi\rangle}\sigma_P(x,y,\xi)$.
\begin{remark}
Lemma \ref{l.jatz} implies that the function $t\mapsto J_A(\cdot,\cdot,t)$ extends continuously to $t=0$ as a function from $\R_+$ to $C^\infty(U\times U)$ and
\begin{equation}\label{e.j_is_y}
  J_A(x,y,0)=\int_{S^*_y}\sigma_P(x,y,\partial_{x,y}(\partial_\xi\psi(x,y,0)\xi))d_y\nu(\xi).
\end{equation}
\end{remark}
\begin{prf}
The first statement follows from the fourth point of Lemma \ref{l.toolbox}. For the second statement, by Equation \eqref{e.scaling}, we may therefore restrict our attention to the case where $\xi\in S^*_y$. Next, we observe that by the second point of Lemma \ref{l.toolbox}, $H_P(y,y,\xi,t)=H_P(y,y,\xi,0)$. The function $H_P$ is clearly $C^1$ with respect to its first variable so that $|H_P(x,y,\xi,t)-H_P(x,y,\xi,0)|$ is no greater than
\[
|x-y|\sup_{s\in[0,1]}|\partial_xH_P(sx+(1-s)y,y,\xi,t)-\partial_xH_P(sx+(1-s)y,y,\xi,0)|\, .
\]
Let us fix $\Omega\subset U$ a compact set. Then by Taylor's inequality, there exists $C_1<+\infty$ such that for each $x,y\in \Omega$, $\xi\in S^*_y$ and each $t>0$,
\[
|\partial_{x,y}\psi(x,y,t\xi)-\partial_{x,y}\psi(x,y,0)-\partial_{x,y}(\partial_\xi\psi(x,y,0)\xi)t|\leq C_1t^2\, .
\]
By the first point of Lemma \ref{l.toolbox}, $\partial_{x,y}\psi(x,y,0)=0$, so that
\begin{equation}\label{jatz_eq_1}
\partial_{x,y}\psi_t(x,y,\xi)=\partial_{x,y}(\partial_\xi\psi(x,y,0)\xi)+O(t)
\end{equation}
uniformly in $x,y\in \Omega$ and $\xi\in S^*_y$. On the other hand by the fifth point of Lemma \ref{l.toolbox}, $(\psi_t)_{t>0}$ is bounded in $C^\infty$. In particular, there exists a constant $C_2<+\infty$ such that for each $t>0$, each $x,y\in K$ and each $\xi\in S^*_y$, $|\psi_t(x,y,\xi)|\leq C_2$. In other words
\begin{equation}\label{jatz_eq_2}
\psi(x,y,t\xi)=O(t)
\end{equation}
uniformly in $x,y\in \Omega$ and $\xi\in S^*_y$. Applying estimates \eqref{jatz_eq_1} and \eqref{jatz_eq_2} to each occurrence of $\psi$ in $H_P$, we see that uniformly for $x,y$ in compact subsets of $U$ and $\xi\in S^*_y$,
\[
\partial_xH_P(x,y,\xi,t)=\partial_xH_P(x,y,\xi,0)+O(t)\, ,
\]
which completes the proof.
\end{prf}

\subsection{Proof of Lemma \ref{l.jati} and its analogue in dimension one}\label{ss.jp}

In this subsection, we use the results of the previous subsection to prove Lemma \ref{l.jati}. We will use this lemma in the proof of the multi-dimensional case of Theorem \ref{t.main.2} (see Section \ref{s.proofs.2}). In the one dimensional case, we will use Lemma \ref{l.decay_one} presented below.

\begin{prf}[ of Lemma \ref{l.jati}]
  To prove this lemma, we interpret $J_A$ as an oscillatory integral whose phase is a deformation of $(\omega,\tau)\mapsto\langle\omega,\tau\rangle$. First, fix $\Omega\subset U$ a compact neighborhood of $0$. Let $r_0>0$ be such that $\Omega_0=\{x\in\R^n\ |\ \exists y\in \Omega,\ |y-x|\leq r_0\}\subset U$. By the fourth point of Lemma \ref{l.toolbox} the family $(\psi_t)_{t>0}$ extends by continuity to $t=0$ in $C^\infty$. For each $t\geq 0$, $y\in U$, $0<r\leq r_0$ and $\xi,\tau\in\R^n$ such that $|\tau|\leq 1$, let
\[
f_{t,r}(y,\xi,\tau)=r^{-1}\psi_t((y+r\tau),y,\xi).
\]
Let $\alpha\in\N^n$. The Taylor expansion of $\partial^\alpha_\xi\psi_t(y+r\tau,y,\xi)$ along $r$ yields, for each $y\in \Omega$, $|\tau|\leq 1$, $0<r\leq r_0$, $t\geq 0$ and $\xi\in S^*_y$,
\[
\left|\partial^\alpha_\xi\psi_t(y+r\tau,y,\xi)-\partial^\alpha_\xi\langle \xi,\tau\rangle\right|\leq \frac{1}{2}C_1r
\]
where
\[
C_1=\sup\{|\partial_x\partial_\xi\psi_s(w',w,\xi)|\ |\ w\in \Omega,\ w'\in \Omega_0,\ \xi\in S^*_w,\ s\geq 0\}\, .
\]
The constant $C_1$ is finite by the fifth point of Lemma \ref{l.toolbox}. In particular,
\[
\lim_{r\rightarrow 0}f_{t,r}(y,\xi,\tau)=\langle\xi,\tau\rangle
\]
smoothly in $\xi$, uniformly in $t\geq 0$, $y\in\Omega$ and $\tau\in\R^n$ such that $|\tau|\leq 1$. In particular, we have proved first that $f_{t,r}(y,\xi,\tau)\xrightarrow[t,r\rightarrow 0]{} \langle \xi,\tau\rangle$ in this same topology, and second that for each $\alpha\in\N^n$, the map $(t,r)\rightarrow \partial^\alpha_\xi f_{t,r}$ is continuous at $(t,0)$ for any $t\geq 0$ for the topology of uniform convergence. Since this map is obviously continuous as long as $r>0$ we have proved that the family $(f_{t,r})_{t,r}$ is a deformation of the height function in the sense of Definition \ref{d.ndls}. Now let $x\in U$ be such that  $0<r:=|x-y|\leq r_0$ and let $\tau =\frac{x-y}{|x-y|}$. Then $|\tau|=1$ and
\[\psi(x,y,t\xi)=t|x-y|f_{t,|x-y|}(y,\xi,\tau).\]
Moreover, by the fifth point of Lemma \ref{l.toolbox}, the function
\[
\xi\mapsto\sigma_P(x,y,\partial_{x,y}\psi_t(x,y,\xi))
\]
is bounded in $C^\infty(\R^n)$ uniformly for $x,y\in \Omega$ and $t\geq 1$. Hence, the fact that the function $\sigma_A$ has $\eps$-non-degenerate level sets (see Definition \ref{d.ndls}) implies the existence an open neighborhood $V\subset U$ of $0$ and a constant $C>0$ such that, uniformly for $x,y\in V$ and $t>0$,
\[\Big|\int_{S^*_y}e^{i\psi(x,y,t\xi)}\sigma_P(x,y,\partial_{x,y}\psi_t(x,y,\xi))d_y\nu(\xi)\Big|\leq C(t|x-y|)^{-\eps}.\]
Here we took $\omega = t|x-y|$ in Equation \eqref{e.decay}.
\end{prf}

In dimension $n=1$, the symbol will never have non-degenerate level sets (in fact they will be discrete). Instead of Lemma \ref{l.jati} we will use the following result.

\begin{lemma}\label{l.decay_one}
Assume that $n=1$. For each compact interval $I\subset U$, there exists $C<+\infty$ such that for each $0<a\leq b$, each $\eta\in\{-1,+1\}$ and each $x,y\in I$
 \begin{equation}\label{e.decay_one}
 \left|\int_{\eta a}^{\eta b}e^{i\psi(x,y,|x-y|^{-1}\eta)}\sigma_P(x,y,|x-y|\partial_{x,y}\psi(x,y,|x-y|^{-1}\eta)\sigma_A(y,\xi)^{-(d+1)/m}d\eta\right|\leq Ca^{-1}\, .
 \end{equation} 
\end{lemma}

\begin{prf}[ of Lemma \ref{l.decay_one}]
Let $I\subset U$ be a compact interval. First of all, since $\sigma_A$ is homogeneous of degree $m$ and $n=1$, there exists a positive function $\varrho\in C^\infty(U)$ such that $\sigma_A(x,\eta)=\varrho(x)|\eta|^m$. Thus, we may replace $\sigma_A(x,\eta)$ by $|\eta|^m$ in Equation \eqref{e.decay_one}. Observe that for each $t,\lambda>0$, $x,y\in U$ and $\eta\in \R$,
\[
\psi_t(x,y,\lambda\eta) = \lambda\psi_{\lambda t}(x,y,\eta)\, .
\]
This Equation, combined with the fifth point of Lemma \ref{l.toolbox} implies that there exists $C<+\infty$ such that for each $x,y\in I$, each $t>0$ and $\eta\in\R$
\[
\left|\partial_{x,y}\psi_t(x,y,\eta)\right|\leq C|\eta|\text{ and } \left|\partial_{x,y}\partial_\xi\psi_t(x,y,\eta)\right|\leq C\, .
\]
Since moreover $\sigma_P$ is homogeneous of degree $d$ in the third variable, we have, uniformly for $x,y\in I$ and for non-zero $\eta\in\R\setminus\{0\}$,
  \begin{align*}
    \sigma_P(x,y,|x-y|\partial_{x,y}\psi(x,y,|x-y|^{-1}\eta))|\eta|^{-d-1}= \sigma_P(x,y,\partial_{x,y}\psi_{|x-y|^{-1}}(x,y,\eta))|\eta|^{-d-1}&=O(|\eta|^{-1})\\
    \partial_\eta[\sigma_P(x,y,|x-y|\partial_{x,y}\psi(x,y,|x-y|^{-1}\eta))|\eta|^{-d-1}]=\partial_\eta[\sigma_P(x,y,\partial_{x,y}\psi_{|x-y|^{-1}}(x,y,\eta))|\eta|^{-d-1}]&=O(|\eta|^{-2}).
  \end{align*}  
  In addition, again uniformly for $x,y\in I$ and non-zero $\eta\in\R\setminus\{0\}$, by Lemma \ref{l.dimension_1}, $\partial_\eta^2[\psi(x,y,|x-y|^{-1}\eta)]=O(|\eta|^{-1})$ and $\partial_\eta[\psi(x,y,|x-y|^{-1}\eta)]$ is bounded from above and below by a positive constant. Now, setting momentarily $u(\eta):=\psi(x,y,|x-y|^{-1}\eta)$ and $v(\eta)=\sigma_P(x,y,|x-y|\partial_{x,y}\psi(x,y,|x-y|^{-1}\eta))|\eta|^{-d-1}$, we have, for any $a,b>0$ such that $a\leq b$,
  \[\int_a^be^{iu(\eta)}v(\eta)d\eta=\Big[\frac{1}{i}e^{iu(\eta)}\frac{v(\eta)}{u'(\eta)}\Big]^{b}_{\eta=a}-\int_a^b\frac{1}{i}e^{iu(\eta)}\left(\frac{v'(\eta)}{u'(\eta)}-\frac{v(\eta)u''(\eta)}{u'(\eta)^2}\right)d\eta.\]
  The preceding observations show that, uniformly for $x,y\in I$, $0<a\leq b$ and $\eta\in[a,b]$, we have $\frac{v(a)}{u'(a)}=O(a^{-1})$, $\frac{v(b)}{u'(b)}=O(b^{-1})$, $\frac{v'(\eta)}{u'(\eta)}=O(\eta^{-2})$ and $\frac{v(\eta)u''(\eta)}{u'(\eta)^2}=O(\eta^{-2})$. Consequently, there exists $C<+\infty$ such that for any $x,y\in \Omega$ and any $0<a\leq b$,
\[\Big|\int_a^be^{i\psi(x,y,|x-y|^{-1}\eta)}\sigma_P(x,y,|x-y|\partial_{x,y}\psi(x,y,|x-y|^{-1}\eta))|\eta|^{-d-1}d\eta\Big|\leq Ca^{-1}.\]
The proof for $\int_{-b}^{-a}$ is identical.
\end{prf}

\section{Proof of Theorem \ref{t.main.3}, Proposition \ref{p.main} and Theorem \ref{t.main.1}}\label{s.proofs.1}

In this section, we prove Theorem \ref{t.main.3}, Proposition \ref{p.main} and Theorem \ref{t.main.1}. We use only Theorem \ref{t.hormander} and Lemma \ref{l.toolbox}.\\

Let $f:\R\rightarrow\R$ with support in $]0,+\infty[$ differentiable almost everywhere. For each $L\geq 1$, let $K_L^f$ be the integral kernel of $\Pi_Lf(A)$. Later in the section, we will be interested in a special case of $K^f_L$. More precisely, we fix $z=z_1+iz_2\in\C$ and set $K_L=K^f_L$ where $f$ is chosen so that $f(t)=t^z$ for $t>0$ large enough. We begin by linking $K_L^f$ with $E_L$.

\begin{lemma}\label{l.link}
For any $L\in\R$,
\[K_L^f=f(L)E_L-\int_0^Lf'(\lambda)E_\lambda d\lambda.\]
\end{lemma}

This lemma generalizes Proposition 21 of \cite{ale161}.

\begin{prf}
The functions $L\mapsto E_L$ and $L\mapsto K_L^f$ are locally constant and define distributions on $\R$ with values in $C^\infty(\man\times\man)$. We denote by $'$ the weak derivative with respect to $L$ of these kernels. For each $u,v\in C^\infty(\man)$ we let $u\boxtimes v\in C^\infty(\man\times\man)$ be the function $(u\boxtimes v)(x,y)=u(x)v(y)$. For all $L>0$,
\[
E_L=\sum_{\lambda_k\leq L} e_k\boxtimes\overline{e_k};\ K_L^f=\sum_{\lambda_k\leq L}f(\lambda_k)e_k\boxtimes\overline{e_k}\, ,
\]
so that
\[
\left(K_L^f\right)'=\sum_{k\in\N}\delta_{\lambda_k}(L)f(\lambda_k)e_k\boxtimes\overline{e_k}=f(L)\sum_{k\in\N}\delta_{\lambda_k}(L)e_k\boxtimes\overline{e_k}=f(L)E_L'
\]
and
\[K_L^f=\int_0^Lf(\lambda)E_\lambda'd\lambda.\]
By integration by parts,
\[K_L^f=f(L)E_L-f(0)E_0-\int_0^Lf'(\lambda)E_\lambda d\lambda = f(L)E_L-\int_0^Lf'(\lambda) E_\lambda d\lambda\]
since $f(0)=0$.
\end{prf}

We can now prove both Theorem \ref{t.main.3} and Proposition \ref{p.main} using Theorem \ref{t.hormander}. We start with Theorem \ref{t.main.3}.

\begin{prf}[ of Theorem \ref{t.main.3}]
Let $L>0$. Then, by Lemma \ref{l.link}, we have, for each $t\geq L$,
\[
K^f_{L+t}=(L+t)^z E_{L+u}-\int_0^{L+t}f'(\lambda)E_\lambda d\lambda\, .
\]
Now, if instead of $K^f_L$ we consider the special case $K_L$, and if we apply the operator $P$, then, for all large enough values of $L>0$ and all $t\geq 0$,
\[
PK_{L+t}-PK_L=(L+t)^zPE_{L+t}-L^zPE_L-\int_L^{L+t}z\lambda^{z-1}PE_\lambda d\lambda\, .
\]
By Theorem \ref{t.hormander}, we have, uniformly for $(x,y)\in U\times U$ and $t\geq 0$,
\[
(L+t)^zPE_{L+t}(x,y)=O\left(L^{z_1+(n+d)/m}\right)
\]
and
\[
\int_L^{L+t}z\lambda^{z-1}E_\lambda d\lambda=O\left(\int_L^{+\infty}\lambda^{-1+z_1+(n+d)/m}d\lambda\right)=O\left(L^{z_1+(n+d)/m}\right)\, .
\]
In particular, uniformly for $(x,y)\in U\times U$ and $t\geq 0$,
\[
PK_{L+t}(x,y)-PK_L(x,y) = O\left(L^{z_1+(n+d)/m}\right)\, .
\]
Since, $z_1+(n+d)/m<0$, this last estimate implies that the sequence $(PK_L)_{L> 0}$ is a Cauchy sequence in $C^0\left(U\times U\right)$. Therefore, it converges uniformly on compact subsets of $U\times U$ to some function $K_\infty^P\in C^0\left(U\times U\right)$. Since this is actually true for any differential operator of order at most $d$ (indeed, if $d'\leq d$, we still have $z_1+(n+d')/m<0$), all the derivatives of $K_L$, of order up to $d$, converge uniformly on compact sets. But this means that the limit $K_\infty$ of $\left(K_L\right)_{L>0}$ is actually of class $C^d$ and that the limits of the respective derivatives converge to the derivatives of the limit. In particular, $K_\infty^P=P K_\infty$.
\end{prf}

We now move on to Proposition \ref{p.main}.

\begin{prf}[ of Proposition \ref{p.main}]
By Theorem \ref{t.hormander}, uniformly for $x,y\in U$ and $L\geq 1$,
\begin{align*}
PE_L(x,y)&=\frac{1}{(2\pi)^n}\int_{\sigma_A(y,\xi)\leq L}e^{i\psi(x,y,\xi)}\sigma_P(x,y,\partial_{x,y}\psi(x,y,\xi))d\xi+O\left(L^{(n+d-1)/m}\right)\\
         &=\frac{1}{(2\pi)^n}\int_0^{L^{1/m}}J_A(x,y,t)t^{n+d-1}dt+O\left(L^{(n+d-1)/m}\right).
\end{align*}
In the second equality we used the definition of $d\nu$ (see \eqref{e.nu}) and $J_A$ (see \eqref{e.osc}) as well as the fact that $\sigma_P$ is $d$-homogeneous along the fibers. Consequently, uniformly for any $x,y\in U$ and $L\geq 1$,
\[-\int_0^Lf'(\lambda)PE_\lambda(x,y) d\lambda=-\frac{1}{(2\pi)^n}\int_0^Lf'(\lambda)\int_0^{\lambda^{1/m}}J_A(x,y,t)t^{n+d-1}dtd\lambda+O\left(\int_{-\infty}^Lf'(\lambda)\lambda^{(n+d-1)/m}d\lambda\right).\]
Integrating by parts along $\lambda$ the first term in the right hand side, we get
\[-f(L)PE_L(x,y)+\frac{1}{(2\pi)^n}\int_0^Lf(\lambda)\frac{1}{m}\lambda^{\frac{1}{m}-1}J_A(x,y,\lambda^{1/m})\lambda^{(n+d-1)/m}d\lambda+O\left(f(L)L^{(n+d-1)/m}\right).\]
Setting $u=\lambda^{1/m}$ we get
\begin{align*}
\int_0^Lf(\lambda)\frac{1}{m}\lambda^{\frac{1}{m}-1}J_A(x,y,\lambda^{1/m})&\lambda^{(n+d-1)/m}d\lambda=\int_0^{L^{1/m}}f(u^m)J_A(x,y,u)u^{n+d-1}du\\
&=\int_{\sigma_A(y,\xi)\leq L}e^{i\psi(x,y,\xi)}f(\sigma_A(y,\xi))\sigma_P(x,y,\partial_{x,y}\psi(x,y,\xi))d\xi.
\end{align*}
By Lemma \ref{l.link},
\[PK_L^f=f(L)PE_L-\int_0^Lf'(\lambda)PE_\lambda d\lambda.\]
Replacing the integral term by the expression derived above, we see that the $f(L)PE_L$ terms cancel out, leaving the equation from the first result of Proposition \ref{p.main}. For the case where $(x,y)\in U\times U\setminus W$, we just apply the corresponding estimate from Theorem \ref{t.hormander} and proceed accordingly.
\end{prf}

For the proof of Theorem \ref{t.main.1}, we remind teh reader that $K_L=K^f_L$ where $f(t)=t^z$ for $t>0$ large enough.

\begin{prf}[ of Theorem \ref{t.main.1}]
Throughout the proof, we let $\eta=1$ if $n+d+mz=1$ and $0$ otherwise and set $g(L)=L^{(n+d-1)/m+z_1}\ln(L)^\eta$. Let $\Omega$ be a compact neighborhood of $0$ in $U$ such that for any $w,x\in \Omega$ and $L\geq 1$, $w+L^{-1/m}x$ belongs to $U$. Firstly, changing $f$ on a compact set affects $PK_L$ by adding a linear combination of smooth functions (independent of $L$). Thus, we may assume that $f(t)=t^z\one[t\geq 1]$. By Proposition \ref{p.main}, uniformly for $w,x,y\in \Omega$ and $L\geq 1$, $PK_L\left(w+L^{-1/m}x,w+L^{-1/m}y\right)$ equals
\begin{align}\label{e.main.1.6}
\frac{1}{(2\pi)^n}&\int_{1\leq \sigma_A(w+L^{-1/m}y,\xi)\leq L}\sigma_A(w+L^{-1/m}y,\xi)^ze^{i\psi(w+L^{-1/m}x,w+L^{-1/m}y,\xi)}\\\nonumber
&\times\sigma_P\left(w+L^{-1/m}x,w+L^{-1/m}y,\partial_{x,y}\psi\left(w+L^{-1/m}x,w+L^{-1/m}y,\xi\right)\right)d\xi+O(g(L))\, .\\\nonumber
\end{align}
Indeed, since $n+d+z_1>0$, $O(g(L))+O(1)=O(g(L))$. We need to check that replacing each occurrence of $w+L^{-1/m}y$ or $w+L^{-1/m}x$ by $w$ in the integrand will produce an error of order $O(g(L))$. More precisely, we make the following claim.
\begin{claim}\label{cl.main.2}
Uniformly for $w,x,y\in\Omega$, $\xi\in\R^n\setminus\{0\}$ and $L\geq 1$ such that $1\leq \sigma_A(w+L^{-1/m}y,\xi)\leq L$, the quantity
\begin{align}\label{e.main.1.1}
\sigma_A&(w+L^{-1/m}y,\xi)^ze^{i\psi(w+L^{-1/m}x,w+L^{-1/m}y,\xi)}\\\nonumber
&\times\sigma_P\left(w+L^{-1/m}x,w+L^{-1/m}y,\partial_{x,y}\psi\left(w+L^{-1/m}x,w+L^{-1/m}y,\xi\right)\right)
\end{align}
equals
\begin{equation}\label{e.main.1.2}
e^{iL^{-1/m}\langle\xi,x-y\rangle}\sigma_A(w,\xi)^z\sigma_P(w,w,(\xi,-\xi))+O\left(|\xi|^{mz_1+d}L^{-1/m}\right)\, .
\end{equation}
\end{claim}
\begin{prf}
Throughout the proof we fix $w,x,y\in\Omega$, $L\geq 1$ and $\xi\in\R^n\setminus\{0\}$ such that $\sigma_A(w+L^{-1/m}y,\xi)\leq L$. Unless otherwise stated, all the $O$ estimates will be uniform with respect to these parameters. First of all, since $\sigma_A$ is a positive $m$-homogeneous symbol in its second variable, $\sigma_P$ is a symbol of order $d$ in its third variable and $\partial_{x,y}\psi$ is a symbol of order $1$ in its third variable, applying Taylor's inequality with respect to the $L$-dependent variables everywhere except the exponential in the quantity \eqref{e.main.1.1} shows that it equals
\begin{equation}\label{e.main.1.4}
\sigma_A(w,\xi)^ze^{i\psi(w+L^{-1/m}x,w+L^{-1/m}y,\xi)}\sigma_P\left(w,w,\partial_{x,y}\psi\left(w,w,\xi\right)\right)+O\left(|\xi|^{mz_1+d}L^{-1/m}\right)\, .
\end{equation}
Here the $|\xi|^{mz_1}$ appears regardless of the sign of $z_1$ because $\sigma_A$ is positive homogeneous. Since $\psi$ is a symbol of order one in $\xi$ and $|\xi|=O\left(L^{1/m}\right)$,
\[
\psi(w+L^{-1/m}x,w+L^{-1/m}y,\xi)=\psi(w,w,\xi)+\partial_x\psi(w,w,\xi)L^{-1/m}x+\partial_y\psi(w,w,\xi)L^{-1/m}y+O\left(L^{-1/m}\right)\, .
\]
By points two and three of Lemma \ref{l.toolbox} we get
\[
\psi(w+L^{-1/m}x,w+L^{-1/m}y,\xi)=L^{-1/m}\langle x-y,\xi\rangle+O\left(L^{-1/m}\right)\, .
\]
Using this estimate in the exponential, together with the fact the rest of the integrand is $O\left(|\xi|^{mz_1+d}\right)$ we obtain that the quantity \eqref{e.main.1.4} equals
\[
e^{iL^{-1/m}\langle\xi,x-y\rangle}\sigma_A(w,\xi)^z\sigma_P(w,w,(\xi,-\xi))+O\left(|\xi|^{mz_1+d}L^{-1/m}\right)
\]
which is exactly \eqref{e.main.1.2}.
\end{prf}
By Claim \ref{cl.main.2} and Equation \eqref{e.main.1.6} $PK_L\left(w+L^{-1/m}x,w+L^{-1/m}y\right)$ equals
\begin{align}\label{e.main.1.3}
PK_L\left(w+L^{-1/m}x,w+L^{-1/m}y\right)=\frac{1}{(2\pi)^n}&\int_{1\leq\sigma_A(w+L^{-1/m}y,\xi)\leq L}e^{iL^{-1/m}\langle\xi,x-y\rangle}\sigma_A(w,\xi)^z\sigma_P(w,w,(\xi,-\xi))d\xi\\\nonumber
&+ O\left(L^{-1/m}\int_{1\leq \sigma_A(w+L^{-1/m}y,\xi)\leq L}|\xi|^{mz_1+d}d\xi\right)\, .
\end{align}
But since $mz_1+d+n>0$ and $\sigma_A$ is $m$-homogeneous, the remainder is $O\left(L^{z_1+(n+d-1)/m}\right)=O(g(L))$. For each $y\in \Omega$ and each $L\geq 1$ let $\Delta(y,L)$ be the symmetric difference of the sets $\{\xi\in\R^n\ |\ 1\leq\sigma_A(w,\xi)\leq L\}$ and $\{\xi\in\R^n\ |\ 1\leq\sigma_A(w+L^{-1/m}y,\xi)\leq L\}$. Since $\sigma_A$ is positive $m$-homogeneous in $\xi$ and smooth in $y$, there exists $0<C<+\infty$ such that for each $L\geq 1$ and $w\in \Omega$, $Vol(\Delta(w,L))\leq C L^{(n-1)/m}$ and for each $\xi\in\Delta(w,L)$, $C^{-1}L^{1/m}\leq |\xi|\leq C L^{1/m}$. Consequently, in Equation \eqref{e.main.1.3} we can replace the integration domain by $\{\xi\in\R^n\ |\ 1\leq\sigma_A(w,\xi)\leq L\}$ and produce an error of order $O\left(L^{z_1+(n+d-1)/m}\right)=O(g(L))$ uniformly for $y\in \Omega$ and $L\geq 1$. In other words,
\[
  PK_L\left(w+L^{-1/m}x,w+L^{-1/m}y\right)=\frac{1}{(2\pi)^n}\int_{1\leq \sigma_A(w,\xi)\leq L}e^{iL^{-1/m}\langle\xi,x-y\rangle}\sigma_A(w,\xi)^z\sigma_P(w,w,(\xi,-\xi))d\xi+O(g(L))\, .
\]
Moreover, since $mz_1+d+n>0$ and the integrand scales like $|\xi|^{mz_1+d}$ near $0$, adding the region $\sigma_A(w,\xi)\leq 1$ to the integration domain creates a bounded error. Following this by the change of variable $\xi=L^{1/m}\zeta$ shows that uniformly for $w,x,y\in\Omega$ and $L\geq 1$
\[
PK_L\left(w+L^{-1/m}x,w+L^{-1/m}y\right)=\frac{1}{(2\pi)^n}\int_{\sigma_A(w,\zeta)\leq 1}e^{i\langle\zeta,x-y\rangle}\sigma_A(w,\zeta)^z\sigma_P(w,w,(\zeta,-\zeta))d\zeta L^{z+(n+d)/m}+O(g(L))\, .
\]
This proves the first statement of the theorem for $V=\mathring{\Omega}$. To prove the second statement, observe that by Lemma \ref{l.link}, uniformly for $L\geq 1$ and $x,y\in \Omega$,
\[
PK_L(x,y)=f(L)PE_L-\int_0^Lf'(\lambda)PE_\lambda(x,y) d\lambda=L^zPE_L(x,y)-\int_1^L\lambda^{z-1}PE_\lambda(x,y)d\lambda+O(1)\,
\]
Next, fix $W\subset V\times V$ a neighborhood of the diagonal. By Theorem \ref{t.hormander}, there exists $C'>0$ such that for any $(x,y)\in (V\times V)\setminus W$ and any $L\geq 1$, $|PE_L(x,y)|\leq C'L^{(n+d-1)/m}$, which implies
\[
|PK_L(x,y)|\leq C'\left(L^{z_1+(n+d-1)/m}+\int_1^L\lambda^{z_1-1+(n+d-1)/m} d\lambda\right)=O(g(L))\, .
\]
This proves the second statement of Theorem \ref{t.main.1}.
\end{prf}

\section{Proof of Theorem \ref{t.main.2}}\label{s.proofs.2}

In this section we prove Theorem \ref{t.main.2}. We use the admissibility condition through Proposition \ref{p.color.1}. Suppose that $n+d+mz=0$, so that $z=-\frac{d+n}{m}$. By Proposition \ref{p.main}, uniformly for $x,y\in U$,
\[PK_L(x,y)=\frac{1}{(2\pi)^n}\int_{\sigma_A(y,\xi)\leq L}e^{i\psi(x,y,\xi)}\sigma_P(x,y,\partial_{x,y}\psi(x,y,\xi))f(\sigma_A(y,\xi))d\xi+O\left(L^{-1/m}\right).\]
 Let $C<+\infty$ be such that  $f(t)=t^z$ for $t>C$. Then,
\begin{align*}
PK_L(x,y)=\frac{1}{(2\pi)^n}\int_{C\leq\sigma_A(y,\xi)\leq L}e^{i\psi(x,y,\xi)}\sigma_P(x,y,\partial_{x,y}\psi(x,y,\xi))\sigma_A(y,\xi)^{-(d+n)/m}d\xi&\\
+Q_1(x,y)+O\left(L^{-1/m}\right)&
\end{align*}
where
\[
Q_1(x,y)=\frac{1}{(2\pi)^n}\int_{\sigma_A(y,\xi)\leq C}e^{i\psi(x,y,\xi)}\sigma_P(x,y,\partial_{x,y}\psi(x,y,\xi))f(\sigma_A(y,\xi))d\xi\, .
\]
We will split the integral term in the last expression of $PK_L$ as follows. For any $x,y\in U$, let
\begin{align*}
I_L(x,y)=\frac{1}{(2\pi)^n}\int_{C\leq \sigma_A(y,\xi)\leq L}\one[\sigma_A(y,\xi)|x-y|^m\geq 1]e^{i\psi(x,y,\xi)}\sigma_P(x,y,\partial_{x,y}\psi(x,y,\xi))\sigma_A(y,\xi)^{-(d+n)/m}d\xi\\
II_L(x,y)=\frac{1}{(2\pi)^n}\int_{C\leq\sigma_A(y,\xi)\leq L}\one[\sigma_A(y,\xi)|x-y|^m< 1]e^{i\psi(x,y,\xi)}\sigma_P(x,y,\partial_{x,y}\psi(x,y,\xi))\sigma_A(y,\xi)^{-(d+n)/m}d\xi.
\end{align*}
Then, uniformly for $x,y\in U$,
\begin{equation}\label{e.main.2.1}
PK_L(x,y)=I_L(x,y)+II_L(x,y)+Q_1(x,y)+O\left(L^{-1/m}\right).
\end{equation}
Theorem \ref{t.main.2} is an easy consequence of the following two lemmas.

\begin{lemma}\label{l.control I}
Let $k_0\in\N$, $k_0\geq 2$. Suppose that either $n=1$ or $\sigma_A$ is $\frac{1}{k_0}$-admissible. There exist an open neighborhood $V\subset U$ of $0\in\R^n$, a function $Q_2\in L^\infty(V\times V)$ and a constant $C<+\infty$ such that for any $x,y\in V$ and $L\geq 1$,
\[
\Big|I_L(x,y)-Q_2(x,y)\Big|\leq C \min\left(L^{-1/k_0m}|x-y|^{-1/k_0},1\right)\, .
\] 
\end{lemma}

In dimension one, we prove the lemma using Lemma \ref{l.decay_one} while in the case of admissible symbols we use Proposition \ref{p.jprop}. This proof is the only place where we use these results.

\begin{lemma}\label{l.control II}
  There exist an open neighborhood $V\subset U$ of $0$ and a constant $C<+\infty$ such that for all $x,y\in V$ and $L\geq 1$,
\[
\left|II_L(x,y)-\frac{1}{(2\pi)^n}Y_P(x,y)\left[\ln\left(L^{1/m}\right)-\ln_+\left(L^{1/m}|x-y|\right)\right]\right|\leq C\, .
\]
Moreover $II_L(x,y)$ is independent of $L$ as long as $L\geq 1$ and $L|x-y|^m\geq 1$.
\end{lemma}

Let us first prove that these lemmas imply Theorem \ref{t.main.2}.

\begin{prf}[ of Theorem \ref{t.main.2}]
Let $V$ be the intersection of the $V$'s appearing in Lemmas \ref{l.control I} and \ref{l.control II}. Firstly, Lemma \ref{l.control I} implies that $I_L(x,y)$ is uniformly bounded for $x,y\in V$ and $L\geq 1$. Secondly, Lemma \ref{l.control II} implies that, uniformly for $x,y\in V$ and $L\geq 1$,
\[
II_L(x,y)=\frac{1}{(2\pi)^n}Y_P(x,y)\left[\ln\left(L^{1/m}\right)-\ln_+\left(L^{1/m}|x-y|\right)\right]+O(1)\, .
\]
Plugging these two estimates in Equation \eqref{e.main.2.1} we get the first point of Theorem \ref{t.main.2}. For the second point, we begin by observing that by Lemma \ref{l.control II}, there exists a bounded function $Q_3\in L^\infty(V\times V)$ such that for each $L\geq |x-y|^{-m}$,
\[
II_L(x,y)=-\frac{1}{(2\pi)^n}Y_P(x,y)\ln\left(|x-y|\right)+Q_3(x,y)\, .
\]
Moreover, if $L\geq |x-y|^{-m}$ then $L^{-1/k_0}|x-y|^{-1/k_0m}\leq 1$ so by Lemma \ref{l.control I}, uniformly for any such $x,y$ and $L$,
\[
I_L(x,y)=Q_2(x,y)+O\left(L^{-1/k_0}|x-y|^{-1/k_0m}\right)\, .
\]
Applying these two estimates to Equation \eqref{e.main.2.1} we deduce that, uniformly for $x,y\in V$ and $L\geq 1$ such that $|x-y|\geq L^{-1/m}$,
\[
PK_L(x,y)=-\frac{1}{(2\pi)^n}Y_P(x,y)\ln\left(|x-y|\right)+Q(x,y)+O\left(L^{-1/k_0}|x-y|^{-1/k_0m}\right)
\]
where $Q=Q_1+Q_2+Q_3\in L^\infty(V\times V)$. This proves the estimate in the second point of Theorem \ref{t.main.2}.
\end{prf}

\begin{prf}[ of Lemma \ref{l.control I}]
Suppose first that $\man$ has dimension $n=1$ and fix $\Omega\subset U$ a compact neighborhood of $0$. For $x\neq y$, setting $\eta=|x-y|\xi$, the integral $I_L(x,y)$ equals
\[
\int_{a(x,y)}^{b(x,y,L)} e^{i\psi(x,y,|x-y|^{-1}\eta)}
\sigma_P\left(x,y,|x-y|\partial_{x,y}\psi\left(x,y,|x-y|^{-1}\eta\right)\right)\sigma_A(y,\eta)^{-(d+1)/m}d\eta.
\]
where $a(x,y)$ and $b(x,y,L)$ are the positive numbers defined by $\sigma_A(y,a(x,y))=\max\left(C|x-y|^m,1\right)$ and $\sigma_A(y,b(x,y,L))=\max\left(|x-y|^mL,1\right)$. Since $\sigma_A$ is elliptic positive homogeneous of degree $m>0$ there exists $C_1>0$ such that for each $x,y\in \Omega$ and $L\geq 1$,
\[
b(x,y,L)\geq C_1 \min\left(|x-y| L^{-1/m}\right)\, .
\]  
By Lemma \ref{l.decay_one}, $I_L(x,y)$ converges to some limit $Q_2(x,y)$ as $L\rightarrow\pfty$ in such a way that the remainder term is $O\left(\min\left(|x-y|^{-1}L^{-1/m},1\right)\right)$. The case where $x=y$ follows by continuity and we have proved the lemma in the one-dimensional case with $V=\mathring{\Omega}$.\\

Suppose now that $n\geq 2$ and $\sigma_A$ is $\frac{1}{k_0}$-admissible for some integer $k_0\geq 2$. By Equations \eqref{e.nu} and \eqref{e.osc}, for any $L\geq 1$ and $x,y\in U$,
\begin{align*}
  I_L(x,y)&=\frac{1}{(2\pi)^n}\int_{C^{1/m}}^{L^{1/m}}\one[|x-y|t\geq 1]J_A(x,y,t)\frac{dt}{t}\\
&=\frac{1}{(2\pi)^n}\int_{C^{1/m}|x-y|}^{L^{1/m}|x-y|}\one[s\geq 1]J_A\left(x,y,|x-y|^{-1}s\right)\frac{ds}{s}
\end{align*}
By Proposition \ref{p.jprop}, there exist an open neighborhood $V\subset U$ of $0$ and a constant $C_3>0$ such that, uniformly for distinct $x,y\in V$ and $t>0$, $|J_A(x,y,t)|\leq C_3\left(|x-y|t\right)^{-1/k_0}$. Therefore, for each $x,y\in V$ and $L>0$,
\begin{align*}
  \left|(2\pi)^nI_L(x,y)-\int_{C^{1/m}|x-y|}^{\pfty}\one[s\geq 1]J_A\left(x,y,|x-y|^{-1}s\right)\frac{ds}{s}\right|&\leq C_3\int_{\max\left(|x-y|L^{\frac{1}{m}},1\right)}^{\pfty}s^{-1-1/k_0}ds\\
  &=\frac{C_3}{k_0}\min\left(1,L^{-1/k_0m}|x-y|^{-1/k_0}\right)\, .
\end{align*}
  By continuity, this stays true for $x=y$. This proves the lemma for $\sigma_A$ admissible with  
\[
Q_2(x,y)=\int_{C^{1/m}|x-y|}^{\pfty}\one[s\geq 1]J_A\left(x,y,|x-y|^{-1}s\right)\frac{ds}{s}\, .
\]
\end{prf}

\begin{prf}[ of Lemma \ref{l.control II}]
The proof of the second statement is obvious from the definition of $II_L$ and the expression $\ln\left(L^{\frac{1}{m}}\right)-\ln_+\left(L^{\frac{1}{m}}|x-y|\right)$. We now prove the first statement. For each $y\in U$ and each $0\leq r_1\leq r_2$, we set
\[
\calA_y(r_1,r_2)=\{\xi\in\R^n\, |\, r_1\leq \sigma_A(y,\xi)\leq r_2\}\, .
\]
Recall that
\[
II_L(x,y)=\frac{1}{(2\pi)^n}\int_{\calA_y(C,L)}\one[\sigma_A(y,\xi)|x-y|^m< 1]e^{i\psi(x,y,\xi)}\sigma_P(x,y,\partial_{x,y}\psi(x,y,\xi))\sigma_A(y,\xi)^{-(n+d)/m}d\xi.\, .
\] 
By Equation \eqref{e.amplitude}, the integrand equals
\[
\one[\sigma_A(y,\xi)|x-y|^m< 1]H_P\left(x,y,\sigma_A(y,\xi)^{-1/m}\xi,\sigma_A(y,\xi)^{1/m}\right)\sigma_A(y,\xi)^{-n/m}\, .
\]
Since $\sigma_A$ is positive homogeneous of degree $m$, $\sigma_A(y,\xi)^{-1/m}\xi$ is uniformly bounded for $y\in\Omega$ and $\xi\in\R^n\setminus\{0\}$. By the second point of Lemma \ref{l.jatz}, uniformly for $x,y\in\Omega$ and $\xi\in\R^n\setminus\{0\}$,
\[
H_P\left(x,y,\sigma_A(y,\xi)^{-1/m}\xi,\sigma_A(y,\xi)^{1/m}\right)=H_P\left(x,y,\sigma_A(y,\xi)^{-1/m}\xi,0\right)+O\left(|x-y|\sigma_A(y,\xi)^{1/m}\right)\, .
\]
Again by $m$-homogeneity and positivity, $|x-y|\one[\sigma_A(y,\xi)|x-y|^m< 1]\sigma_A(y,\xi)^{(1-n)/m}$ is uniformly integrable in $\xi$ for $x,y\in\Omega$ so
\[
II_L(x,y)=\frac{1}{(2\pi)^n}\int_{\calA_y(C,L)}\one[\sigma_A(y,\xi)|x-y|^m< 1]H_P\left(x,y,\sigma_A(y,\xi)^{-1/m}\xi,0\right)\sigma_A(y,\eta)^{-n/m}d\xi+O(1)\, .
\]
Fix two distinct points $x,y\in U$. The change of variables $\eta=|x-y|\xi$ in the integral yields
\[
\int_{|x-y|\calA_y(C,L)}\one[\sigma_A(y,\eta)<1]H_P(x,y,|x-y|\eta,0)\sigma_A(y,\eta)^{-n/m}d\eta
\]
which, by definition of $J_A$ (see Equation \eqref{e.osc}), equals
\[
J_A(x,y,0)\int_{C^{1/m}|x-y|}^{L^{1/m}|x-y|}\one[|x-y|s<1]\frac{ds}{s}\, .
\]
Observe that for any $0<a\leq b$,
\[
\int_a^b\one[t<1]\frac{dt}{t}=\ln(b)-\ln_+(b)-\ln(a)+\ln_+(a)
\]
where $\ln_+(s)=\max(\ln(s),0)$. In our setting, uniformly for distinct $x,y\in\Omega$,
\[
\int_{C^{1/m}}^{L^{1/m}}\one[|x-y|s< 1]\frac{ds}{s}=\int_{C^{1/m}|x-y|}^{L^{1/m}|x-y|}\one[t< 1]\frac{dt}{t}=\ln\left(L^{1/m}\right)-\ln_+\left(L^{1/m}|x-y|\right)+O(1)\, .
\]
Hence, uniformly for any $(x,y)\in \Omega\times \Omega$ and $L\geq 1$,
\[
II_L(x,y)=\frac{1}{(2\pi)^n}J_A(x,y,0)\left[\ln\left(L^{1/m}\right)-\ln_+\left(L^{1/m}|x-y|\right)\right]+O(1).\]
Finally, by Equation \eqref{e.j_is_y} $J_A(x,y,0)=Y_P(x,y)$ so the lemma is proved with $V=\mathring{\Omega}$. 
\end{prf}

\section{Admissible symbols}\label{s.singularities}

In this section, we deal with results concerning admissible symbols (see Definition \ref{d.admissible}). These results are useful in the proofs of Theorems \ref{t.fat} and \ref{t.main.2}. More precisely, in Subsection \ref{ss.color} we prove Proposition \ref{p.color.1} which says that admissible symbols have non-degenerate level sets and is used in the proof of Theorem \ref{t.main.2}. Then, in Proposition \ref{p.fat} of Subsection \ref{ss.whitney} we prove that admissibility is both stable and generic in a suitable topology. Theorem \ref{t.fat} follows directly from Proposition \ref{p.fat}.

\subsection{Proof of Proposition \ref{p.color.1}}\label{ss.color}

The object of this subsection is to prove Proposition \ref{p.color.1}. To prove this result, we will use partitions of unity and local charts to carry the integral onto $\R^n$ and then apply the following lemma, which we prove later in the section.

\begin{lemma}\label{l.color}
Let $n\in\N$, $n\geq 1$. Let $U\subset\R^n$ be an open neighborhood of $0$ and $(f_\eta)_{\eta\in E}$ be a continuous family of smooth functions on $U$ indexed by $E\subset\R^p$, an open neighborhood of $0$. Fix $k\geq 1$ and assume that $d^k_0f_0\neq 0$. Then, there exist $E'\subset E$ and $U'\subset U$ two open neighborhoods of the origin in $\R^p$ and $\R^n$ respectively, such that for each $u\in C^\infty_c(U')$ there exists $C(u)<+\infty$ such that for each $\lambda>0$ and each $\eta\in E'$,
\[
\left|\int_{U'} e^{i\lambda f_\eta(x)}u(x)dx\right|\leq C(u) \lambda^{-\frac{1}{k}}\, .
\]
Moreover, $C(u)$ depends continuously on $u$ in the $C^\infty_c(U')$ topology.
\end{lemma}

We now begin the proof of Proposition \ref{p.color.1}.

\begin{prf}[ of Proposition \ref{p.color.1}]
Take $\Omega$, $\gamma$, $(f_\eta)_\eta$ and $(u_\eta)_\eta$ as in Definition \ref{d.ndls}. Recall that $d_x\mu$ is the area measure on $S^*_x$. By using partitions of unity on $\R^n$, we may fix  $\xi_0\in\R^n\setminus\{0\}$ and assume that the functions $u_\eta$ are supported near $\xi_0$. Let $\xi_1,\dots,\xi_{n-1}\in\R^n$ be such that $(\xi_0,\xi_1,\dots,\xi_{n-1})$ forms a basis for $\R^n$. For any $x\in S^*U$, let
\[
\beta_x:(t_1,\dots,t_{n-1})\in\R^{n-1}\mapsto \sigma(x,\xi_0+t_1\xi_1+\dots+t_{n-1}\xi_{n-1})^{-\frac{1}{m}}(\xi_0+t_1\xi_1+\dots+t_{n-1}\xi_{n-1})\in S^*_x\, .
\]
The map $\beta_x$ defines a local coordinate system at $\sigma(x,\xi_0)^{-\frac{1}{m}}\xi_0\in S^*_x$. Moreover, the map $x\mapsto \beta_x\in C^\infty(\R^{n-1})$ is continuous. The density $g_x=\frac{\beta_x^*(\gamma(x,\cdot)d_x\mu)}{dt}\in C^\infty(\R^{n-1})$ also depends continously on $x$ in $C^\infty(\R^{n-1})$. Now, for any $\lambda>0$, $\eta\in E$ and $(x,\tau)\in \Omega$, if $u_\eta$ is supported close enough to $\xi_0$,
\[
\int_{S^*_x}e^{i\lambda f_\eta(x,\tau,\xi)}u_\eta(\xi)\gamma(x,\xi)d_x\mu(\xi)=\int_{\R^{n-1}} e^{i\lambda f_\eta(x,\tau,\beta_x(t))}u_\eta(\beta_x(t))g_x(t)dt\, .
\]
We now set $\tilde{E}=U\times\R^n\times E$, for any $\tilde{\eta}=(x,\tau,\eta)\in\tilde{E}$, $\tilde{f}_{\tilde{\eta}}=f_\eta(x,\tau,\beta_x(\cdot))\in C^\infty(\R^{n-1})$ and $\tilde{u}_{\tilde{\eta}}=u_\eta(\beta_x(\cdot))g_x\in C^\infty(\R^{n-1})$. By compactness, it is enough to fix $(x_0,\tau_0)\in \Omega$ and prove estimate \eqref{e.decay} for $\tilde{\eta}=(x,\eta,\tau)$ close enough to $\tilde{\eta}_0=(x_0,0,\tau_0)$. Also, without loss of generality, we may assume $x_0=0$. Our task is therefore to find $C>0$ such that for each $\tilde{\eta}$ close enough to $\tilde{\eta}_0$ and each $\lambda>0$,
\[
\left|\int_{\R^{n-1}}e^{i\lambda \tilde{f}_{\tilde{\eta}}(t)}\tilde{u}_{\tilde{\eta}}(t)dt\right|\leq C\lambda^{-\frac{1}{k_0}}\, .
\]
We wish to apply Lemma \ref{l.color}. The estimate is obvious for $\lambda \leq 1$ while, for $\lambda \geq 1$, replacing $k_0$ by some smaller integer would improve the estimate. Thus, we need only to check that there exists $k\in\{1,\dots,k_0\}$ such that
\begin{equation}\label{e.color.5}
d^k_0\tilde{f}_{\tilde{\eta}_0}\neq 0\, .
\end{equation}
Let $g=\tilde{f}_{\tilde{\eta}_0}$. Since $f_0(x,\tau_0,\xi)=\langle\tau_0,\xi\rangle$, we have, for all $t\in\R^{n-1}$,
\[
g(t)=(\langle \tau_0,\xi_1\rangle t_1+\dots+\langle \tau_0,\xi_{n-1}\rangle t_{n-1}+\langle\tau_0,\xi_0\rangle)\sigma(0,\xi_0+t_1\xi_1+\dots+t_{n-1}\xi_{n-1})^{-\frac{1}{m}}\, .
\]
We proceed by contradiction and assume that $d^j_0g=0$ for each $j\in\{1,\dots,k\}$. To understand how this condition affects $\sigma$ we use the following claim which we prove at the end.

\begin{claim}\label{cl.color.1}
Let $U\subset\R^n$ be an open neighborhood of $0$ and $f\in C^\infty(U)$ be positive valued. Let $\alpha\in\R\setminus\{0\}$ and $k\in\N$ such that $k\geq 1$. Assume that there exist $b\in\R$ and $\tau\in\R^n$ such that $(\tau,b)\neq (0,0)$ such that, writing $h:x\in\R^n\mapsto \langle\tau,x\rangle+b\in\R$ we have, for each $j\in\{1,\dots,k\}$,
\begin{equation}\label{e.color.1}
d^j_0[hf^\alpha]=0\, .
\end{equation}
Then,
\begin{equation}\label{e.color.2}
f(0)^{k-1}d^k_0f = (\alpha+1)(2\alpha+1)\dots((k-1)\alpha+1)(d_0f)^{\otimes k}\, .
\end{equation}
\end{claim}

We wish to use this claim with $\alpha=-\frac{1}{m}$, $h(t)=\langle \tau_0,\xi_1\rangle t_1+\dots+\langle \tau_0,\xi_{n-1}\rangle t_{n-1}+\langle\tau_0,\xi_0\rangle$ and $f(t)=\sigma(0,\xi_0+t_1\xi_1+\dots+t_{n-1}\xi_{n-1})$. In order to apply it, the only thing to check is that $h$ is not identically $0$. But $h=0$ would imply that $\langle \tau_0,\xi_0\rangle=\dots=\langle\tau_0,\xi_{n-1}\rangle=0$. This cannot happen since $\tau_0\neq 0$. Hence, by Claim \ref{cl.color.1} we have the following equality between (symmetric) $k$-forms on the hyperplane $H$ spanned by $(\xi_1,\dots,\xi_{n-1})$,
\begin{equation}\label{e.color.8}
\sigma(0,\xi_0)^{k-1}\partial^k_\xi\sigma(0,\xi)=C(m,k)(\partial_\xi\sigma(0,\xi_0))^{\otimes k}
\end{equation}
where
\[
C(m,k)=\left(-\frac{1}{m}+1\right)\cdots\left(-\frac{k-1}{m}+1\right)=\frac{m(m-1)\dots(m-k+1)}{m^k}\, .
\]
Next, we make the following claim, which we prove at the end.
\begin{claim}\label{cl.color}
Let $m$ be a positive real number and let $f\in C^\infty(\R^p\setminus\{0\})$ be a real-valued $m$-homogeneous function. Then, for each $x\in\R^p\setminus\{0\}$, each hyperplane $H\subset\R^p$ not containing $x$ and each $k_0\geq 2$,
\begin{equation}\label{e.color.9}
\forall k\in\{2,\dots,k_0\},\ f(x)^{k-1}d^k_xf=\frac{m(m-1)\dots(m-k+1)}{m^k}(d_xf)^{\otimes k}
\end{equation}
is equivalent to
\begin{equation}\label{e.color.10}
\forall k\in\{2,\dots,k_0\},\ f(x)^{k-1}d^k_xf\big|_H=\frac{m(m-1)\dots(m-k+1)}{m^k}(d_xf)^{\otimes k}\big|_H\, .
\end{equation}
\end{claim}
This claim implies that $\sigma$ actually satisfies Equation \eqref{e.color.8} on the whole of $T^*_\xi\R^n\simeq\R^n$. By the assumption on $\sigma$, this Equation cannot be satisfied for all $k\leq k_0$. Hence, $d^k_0g$ cannot vanish for each $k\in\{1,\dots,k_0\}$. In particular, there exists $k\in\{1,\dots,k_0\}$ for which $\tilde{f}_{\tilde{\eta}}$ satisfies Equation \eqref{e.color.5}. Hence, Lemma \ref{l.color} applies for this $k$ and we are done.

\begin{prf}[ of Claim \ref{cl.color.1}]
Let $f$, $\tau$, $b$, $\alpha$, $h$ and $k$ be as in the statement of the claim. Let $g(x)=h(x)^{-\frac{1}{\alpha}}$. First of all, by Equation \eqref{e.color.1} with $j=1$,
\[
f(0)\tau = -\alpha b d_0f
\]
In particular, since $(\tau,b)\neq 0$ and $f(0)>0$, we actually have $b\neq 0$. Thus, the function $g:x\mapsto h(x)^{-\frac{1}{\alpha}}$ is well defined and positive near the origin. Moreover, $hg^\alpha=1$ so all of its derivatives vanish. Consequently, for each $j\in\{1,\dots,k\}$, $d^j_0(f^\alpha g^{-\alpha})=0$ which in turn gives, for each $j\in\{1,\dots,k\}$, $d^j_0(fg^{-1})=0$ (here we use the fact that $fg^{-1}=\left(f^\alpha g^{-\alpha}\right)^{\frac 1\alpha}$ which is well defined near $0$). In particular, the Taylor expansions of $f$ and $g$ coincide to the $k$th order up to a multiplicative constant. By homogeneity of Equation \eqref{e.color.2} we may assume that they agree up to order $k$. But
\begin{align*}
d^k_0g &= \prod_{j=0}^{k-1}\left(-\frac{1}{\alpha}-j\right)\times b^{-\frac{1}{\alpha}-k}\tau^{\otimes k}\\
&=\left(b^{-\frac{1}{\alpha}}\right)^{1-k}(\alpha+1)(2\alpha+1)\dots((k-1)\alpha+1)\left(-\alpha b^{\frac{1}{\alpha}+1}\right)^{-k}\tau^{\otimes k}
\end{align*}
and $g(0)=b^{-\frac{1}{\alpha}}$ and $d_0g = \left(-\alpha b^{\frac{1}{\alpha}+1}\right)^{-1}\tau$.
Thus,
\[
g(0)^{k-1}d^k_0g =(\alpha+1)(2\alpha+1)\dots((k-1)\alpha+1)(d_0g)^{\otimes k}\, .
\]
Since $f$ agrees with $g$ up to order $k$, $f$ satisfies Equation \eqref{e.color.2}.
\end{prf}
\begin{prf}[ of Claim \ref{cl.color}]
Equation \eqref{e.color.9} implies \eqref{e.color.10} by restriction to $H$. Let us assume \eqref{e.color.10} and prove the converse. Since $x\notin H$, $\R x\bigoplus H$ generate $\R^p$. By multilinearity, it is enough to prove \eqref{e.color.9} when the $k$ forms are evaluated on families of the form $(x,\dots,x,y_1,\dots,y_h)$ where $y_1,\dots,y_h\in H$ and $h\in\{1,\dots,k\}$. Now, since $f$ is homogeneous, by Euler's Equation, for any $h\leq k$, and for any $y_1,\dots,y_h\in H$,
\begin{align*}
d^k_xf(x,\dots,x,y_1,\dots,y_h)&=\underbrace{(m-h)\dots(m-k+1)}_{1\text{ if }k=h}d^h_xf(y_1,\dots,y_h)\\
\text{and }(d_xf)^{\otimes k}(x,\dots,x,y_1,\dots,y_h)&=m^{k-h}f^{k-h}(x)(d_xf)^{\otimes h}(y_1,\dots,y_h).
\end{align*}
Applying \eqref{e.color.10} to compare the right hand sides of each line we get Equation \eqref{e.color.9}.
\end{prf}
\end{prf}

The proof of Lemma \ref{l.color} will combine two theorems from singularity theory and oscillatory integral asymptotics which we state now.\\

The following theorem is a corollary of the Malgrange preparation theorem presented in \cite{ho_apdo1}. We give a slightly different formulation and add the continuity with respect to smooth perturbations, which actually follows from H\"ormander's original proof.
\begin{theorem}[\cite{ho_apdo1}, Theorem 7.5.13]\label{t.color.1}
Let $U\subset\R\times\R^n$ (resp. $E\subset\R^p$) be an open neighborhood of $0\in\R\times\R^n$ (resp. $0\in\R^p$) and $(f_\eta)_{\eta\in E}$ be a continuous family of smooth functions on $U$. We denote by $(t,x)$ the elements of $U$. Let $k\in\N$, $k\geq 2$. Assume that for each $\eta\in E$ and $j\in\{0,\dots,k-1\}$
\[
\partial^j_tf_\eta(0,0)=0
\]
and that $\partial^k_tf_\eta(0,0)>0$. Then, there exist $W\subset\R\times\R^n$ (resp. $V\subset\R^n$) a neighbohood of $0\in\R\times\R^n$ (resp. $0\in\R^n$) with $U'\subset \R\times V$ such that for each $\eta\in E$, there exist $\phi_\eta\in C^\infty(W)$ as well as $a^1_\eta,\dots,a^{k-1}_\eta\in C^\infty(V)$, satisfying, for any $\eta\in E$, $(t,x)\in W$,
\begin{align*}
\phi_\eta(0,0)&=0\, ,\\
\partial_t\phi_\eta(0,0)&>0\, ,\\
a^1_\eta(0)=\dots=a^{k-1}_\eta(0)&=0\, ,\\
\text{and }f_\eta(\phi_\eta(t,x),x)&=t^k+\sum_{j=0}^{k-1}a_\eta^j(x)t^j\, .
\end{align*}
Moreover, one can choose these functions such that the maps $\eta\mapsto \phi_\eta$ and $\eta\mapsto a_\eta^j$ are continuous into $C^\infty$.
\end{theorem}
\begin{prf}
First, apply Theorem 7.5.13 of \cite{ho_apdo1} to each $f_\eta$ and define $\tilde{\phi}_\eta(\cdot,x)$ as the inverse map of $T(\cdot,x)$ for the $T$ corresponding to $f_\eta$. That the maps $\tilde{\phi}_\eta$ and $\tilde{a^j}_\eta$ depend continuously on $\eta$ follows from the proof of the aforementioned result. Indeed, they are built as solutions of ODEs whose initial conditions depend continuously on $f$ in $C^\infty$. Finally, by rescaling the new variable $t$ and thus replacing $\tilde{\phi}_\eta$ (resp. $\tilde{a^j}_\eta$) by $\phi_\eta$ (resp. $a^j_\eta$) we get rid of the $\frac{1}{k}$ factor appearing in front of $T^k$ in Theorem 7.5.13 of \cite{ho_apdo1}.
\end{prf}

The following theorem is the special case of Theorem 4 of \cite{col77} (and the remarks 2.3 and 2.4 that follow it) of type $A_n$ singularities. 
\begin{theorem}[\cite{col77}, Theorem 4]\label{t.color.2}
Let $k\in\N$, $k\geq 2$. There exist $\delta=\delta(k)>0$, $V\subset\R^k$ an open neighborhood of $0$ such that for all $u\in C^\infty_c(]-\delta,\delta[)$ there exists $C(u)<+\infty$ such that for all $\lambda>0$ and $(a_0,\dots,a_{k-1})\in V$,
\[
\left|\int_\R e^{i\lambda(t^k+a_{k-1}t^{k-1}+\dots+a_0)}u(t)dt\right|\leq C(u)\lambda^{-\frac{1}{k}}\, .
\]
Moreover $C(u)$ depends continuously on $u\in C^\infty_c(]-\delta,\delta[)$.
\end{theorem}
\begin{prf}
In the terminology of \cite{col77}, the map $(t,a_0,\dots,a_{k-1})\mapsto t^k+a_{k-1}t^{k-1}+\dots+a_0$ is the universal unfolding of the singularity type $A_{k-1}$. In the notations of \cite{col77} our $k$ corresponds to their $n$ while their $k$ equals $1$ in our setting. Moreover, as stated in the table preceding Theorem 4 of \cite{col77}, in the case $A_{k-1}$, $\eps(\sigma)=\frac{1}{2}-\frac{1}{k}$ so that the integral decays polynomially in $\lambda$ at order $-\frac{1}{2}+\eps(\sigma)=-\frac{1}{k}$. 
\end{prf}

\begin{prf}[ of Lemma \ref{l.color}]
Let $(f_\eta)_\eta$ and $k\geq 1$ be as in the statement of the lemma. To make use of the assumption $d^k_0f_0\neq 0$ we use the following elementary result in multilinear algebra which we prove at the end.

\begin{claim}\label{cl.color.2}
Let $\omega$ be a symmetric $k$-linear form on $\R^n$. Let $q:\R^n\rightarrow\R$ be defined as $q(x)=\omega(x,\dots,x)$. Then, $q=0$ implies $\omega=0$.
\end{claim}

By Claim \ref{cl.color.2} there exists $v\in\R^n\setminus\{0\}$ such that $d^k_0f_0(v,v,\dots,v)\neq 0$. Without loss of generality, we may assume that $v=e_n:=(0,\dots,0,1)$. We write $x=(\tilde{x},x_n)\in\R^n=\R^{n-1}\times\R$. Let $u\in C^\infty_c(U)$ be such that $u(\tilde{x},x_n)\neq 0$ implies that $\|\tilde{x}\|_{\infty}\leq 1$. Then, for each $\eta\in E$ and $\lambda>0$,
\[
\left|\int_Ue^{i\lambda f_\eta(x)}u(x)dx\right|\leq \max_{\tilde{x}\in\R^{n-1}}\left|\int_\R e^{i\lambda f_\eta(\tilde{x},x_n)}u(\tilde{x},x_n)dx_n\right|\, .
\]
This way by replacing $\eta$ by $(\eta,\tilde{x})$ and $f_\eta$ by $f_\eta(\tilde{x},\cdot)$ we have reduced the problem to the one dimensional case. From now on, we assume that $n=1$.\\

For each $\eta\in E$, each $x\in U$, and $q=(q_0,\dots,q_{k-1})\in\R^k$, let
\[
g_\eta(x,q)=f_\eta(x)-f_\eta(0)-f_\eta'(0)x-\dots -\frac{1}{(k-1)!}f_\eta^{(k-1)}(0)x^{k-1}+q_0+a_1x+\dots+q_{k-1} x^{k-1}\, .
\]
We will first prove the desired bound where we replace $f_\eta$ by  $g_\eta(\cdot,q)$, uniformly for $q$ close enough to $0$ and then deduce the result for $f_\eta$ itself as a phase.\\

The map $\eta\mapsto g_\eta$ is continuous from $E$ to $C^\infty(U\times \R^k)$. Moreover, for each $\eta\in E$ close enough to $0$ we have
\begin{align*}
\forall j\in\{0,\dots,k-1\},\, \partial_x^jg_\eta(0,0)&=0\\
\partial_x^kg_\eta(0,0)\neq 0.
\end{align*}
Replacing $f_\eta$ by $-f_\eta$ does not change the estimate since it amounts to complex conjugation of the integrand. With this in mind, we may assume that $\partial_x^kg_\eta(0,0)>0$. By Theorem \ref{t.color.1}, there exist $W\subset\R\times\R^k$ and continuous families of smooth functions $(a^1_\eta)_\eta\dots,(a^{k-1}_\eta)_\eta$, as well as $(\phi_\eta)_\eta$ defined respectively in a neighborhood of $0\in\R^k$ and a neighborhood of $(0,0)$ in $U\times\R^k$ such that for each $\eta\in E$ and $(x,q)$ close enough to $0$ and $(0,0)$ respectively,
\begin{align*}
\phi_\eta(0,0)&=0\, ,\\
\partial_t\phi_\eta(0,0)&>0\, ,\\
a^1_\eta(0)=\dots=a^{k-1}_\eta(0)&=0\, ,\\
\text{and }g_\eta(\phi_\eta(x,q),q)&=x^k+\sum_0^{k-1}a^j_\eta(q)x^j\, .
\end{align*}
Hence, if $u\in C^\infty_c(\R)$ is supported close enough to $0$, we have, for all $\eta\in E$ close enough to $0$ and all $q\in\R^k$,
\[
\int_\R e^{i\lambda g_\eta(y,q)}u(y)dy=\int_\R e^{i\lambda(x^k+a^{k-1}_\eta(q) x^{k-1}+\dots + a^0_\eta(q))}u(\phi_\eta(x,q))\left(\phi_\eta^{-1}(\cdot,q)\right)'(x)dx\, .
\]
By Theorem \ref{t.color.2}, there exist $W_1\subset \R^k\times E$ a neighborhood of $0$ such that $\delta>0$ such that for each $(q,\eta)\in W_1$, for each $v\in C^\infty_c\left(]-\delta,\delta[\right)$, there exists $C'(v)<+\infty$ such that for each $\lambda>0$ and each $(q,\eta)$ close enough to $(0,0)$,
\[
\left|\int_\R e^{i\lambda(x^k+a^{k-1}_\eta(q)x^{k-1}+\dots+a_\eta^0(q))}v(x)dx\right|\leq C'(v)\lambda^{-\frac{1}{k}}\, .
\]
Moreover, Theorem \ref{t.color.2} specifies that the map $v\in C^\infty_c\left(]-\delta,\delta[\right)\rightarrow C'(v)\in\R$ is continuous. By continuity, there exist $\eps>0$ and $W_2\subset W_1$ a compact neighborhood of $0$ such that for any $(q,\eta)\in W_2$ and any $x\in\R$ with $|x|\geq\delta/2$, $|\phi_\eta(x,q)|\geq\eps$. In particular, the map $(q,\eta,u)\in W_2\times C^\infty_c\left(]-\eps,\eps[\right)\mapsto u(\phi_\eta(\cdot,q))\left(\phi_\eta(\cdot,q)^{-1}\right)'\in C^\infty_c\left(]-\delta,\delta[\right)$ is well defined and continuous. Consequently, so is the map
\begin{align*}
W_2\times C^\infty_c\left(]-\eps,\eps[\right)&\rightarrow\R\\
(q,\eta,u)&\mapsto C_{q,\eta}(u)=C'\left(u(\phi_\eta(\cdot,q))\left(\phi_\eta(\cdot,q)^{-1}\right)'\right)\, .
\end{align*}
By compactness, $C(u)=\sup_{(q,\eta)\in W_2}C_{q,\eta}(u)$ is finite and continuous in $u$. We have proved that for any $(q,\eta)\in W_2$, any $\lambda>0$ and any $u\in C^\infty_c(]-\eps,\eps[)$,
\[
\left|e^{i\lambda g_\eta(y,q)}u(y)dy\right|\leq C(u)\lambda^{-\frac{1}{k}}\, .
\]
To obtain the corresponding estimate with $f_\eta$ instead of $g_\eta(\cdot,q)$, we make the following two observations. First, for each $\eta\in E$, and $x\in U$,
\[
g_\eta(x,f_\eta(0),\dots,f_\eta^{(k-1)}(0))=f_\eta(x)\, .
\]
Second, since $f_0(0)=\dots=f^{(k-1)}_0(0)=0$, there exists $E'\subset E$ a neighborhood of $0$ such that for each $\eta\in E'$, $(f_\eta(0),\dots,f_\eta^{(k-1)}(0),\eta)\in W_3$. Thus, for each $\eta\in E'$  each $u\in C^\infty_c(]-\eps,\eps[)$ and each $\lambda>0$,
\[
\left|e^{i\lambda f_\eta(y)}u(y)dy\right|\leq C(u)\lambda^{-\frac{1}{k}}
\]
and the proof is over, save for the proof of Claim \ref{cl.color.2}.

\begin{prf}[ of Claim \ref{cl.color.2}]
Let us prove the following formula.
\[
\forall x_1,\dots,x_k\in\R^n,\, \omega(x_1,\dots,x_k)=\frac{1}{2^k}\sum_{\eta\in\{-1,1\}^k}\prod_{i=1}^k\eta_i\times q\left( \sum_{j=1}^k\eta_jx_j\right)\, .
\]
For each $x_1,\dots,x_k\in\R^n$ and $p=(p_1,\dots,p_k)\in\N^k$ such that $p_1+\dots+p_k=k$, we denote by $\omega\left(x_1^{p_1}\dots x_k^{p_k}\right)$ the form $\omega$ evaluated in any $k$-uple with exactly $p_j$ occurrences of $x_j$ ($\forall j\in\{1,\dots,k\}$). Then, for each $x_1,\dots,x_k\in\R^n$,
\begin{align*}
\sum_{\eta\in\{-1,1\}^k}\prod_i^k\eta_i\times q\left(\sum_{j=1}^k\eta_jx_j\right)&=\sum_{\eta\in\{-1,1\}^k}\prod_i^k\eta_i\times\sum_{p_1+\dots+p_k=k}\binom{k}{p_1,\dots,p_k}\omega\left((\eta_1x_1)^{p_1}\dots (\eta_kx_k)^{p_k}\right)\\
&=\sum_{p_1+\dots+p_k=k}\binom{k}{p_1,\dots,p_k}\omega\left(x_1^{p_1}\dots x_k^{p_k}\right)\sum_{\eta\in\{-1,1\}^k}\prod_i^k\eta_i^{p_i+1}\, .
\end{align*}
Given $j\in\{1,\dots,j\}$ and $(p_1,\dots,p_k)$ such that $p_j=0$, applying the bijection
\[
(\eta_1,\dots,\eta_k)\mapsto (\eta_1,\dots,-\eta_j,\dots,\eta_k)
\]
shows that $\sum_{\eta\in\{-1,1\}^k}\prod_j\eta_j^{p_j+1}=0$. Thus, the only remaining term is the one corresponding to $p_1=\dots=p_k=1$ for which the sum of products of the $\eps_j^{p_j+1}$ equals $2^k$. Therefore,
\[
\sum_{\eta\in\{-1,1\}^k}\prod_i\eta_iq\left(\sum_i\eta_iv_i\right)=2^k\omega(v_1,\dots,v_k)
\]
as announced.
\end{prf}
\end{prf}

\subsection{Genericity and stability of the non-degeneracy condition}\label{ss.whitney}

The goal of this subsection is to prove Proposition \ref{p.fat} below, which says roughly that admissible symbols are stable and generic. To give a precise meaning to this statement, we first need to define a topology on the set of positive homogeneous symbols.

\begin{definition}
Fix $n\in\N$, $n\geq 1$ and $m\in]0;+\infty[$. For each $U\subset\R^n$, let $S^m_h(U)\subset C^\infty(U\times (\R^n\setminus\{0\}))$ be the set of smooth functions $m$-homogeneous in the second variable.  We write $S^m_{h,+}(U)$ for the set of positive valued functions in $S^m_h(U)$. The map
\[
S^m_h(U)\rightarrow C^\infty(U\times S^{n-1})\, ,
\]
restricting the second variable to the unit sphere, is a bijection. We endow $S^m_h(U)$ with the topology induced by the Whitney topology on $C^\infty(U\times S^{n-1})$ (see Definition 3.1 of Chapter II of \cite{gg_stable}).
\end{definition}

We have the following proposition.

\begin{proposition}\label{p.fat}
For all $n\in\N$, $n\geq 2$ we define $k_0=k_0(n)\in\N$ as follows. We set $k_0(2)=5$, $k_0(3)=3$, $k_0(4)=3$ and $\forall n\geq 5,\, k_0(n)=2$. Fix $n\geq 2$ and $m>0$. Let $U\subset\R^n$ be an open subset. Then, the set of $\sigma\in S^m_{h,+}(U)$ such that for each $(x,\xi)\in U\times(\R^n\setminus\{0\})$ there exists $j\in\{2,\dots,k_0\}$ such that
\begin{equation}\label{e.fat}
\sigma^{j-1}(x,\xi)\partial_\xi^j\sigma(x,\xi)\neq \frac{m(m-1)\dots(m-j+1)}{m^j}\left(\partial_\xi\sigma(x,\xi)\right)^{\otimes j}
\end{equation}
 is open and dense in $S^m_{h,+}(U)$.
\end{proposition}

To prove this proposition, we will apply Thom's transversality theorem (see Theorem 4.9 of Chapter II of \cite{gg_stable}) to a well chosen submanifold of the jet bundle of $U\times S^{n-1}$ whose codimension grows with the degree of admissibility we consider. Lemmas \ref{l.jet_restriction_map}, \ref{l.jets_of_homogeneous_functions}, \ref{l.jet_equations}, \ref{l.jets_back_home} and \ref{l.jets_we_want} below are devoted to the construction of this manifold. The proof of Proposition \ref{p.fat} is presented only after these are stated and proved. Throughout the rest of the section we fix $n\in\N$, $n\geq 2$, $U\subset\R^n$ an open subset and $m\in\R$, $m>0$. We start by introducing some notation.

\begin{notation}
\begin{enumerate}
\item For each $j,p\in\N$, $p\geq 1$, let $\Sym_p^j$ be the space of symmetric $j$-linear forms over $\R^p$. This is a vector space of dimension $\binom{p+j-1}{j}$. We adopt the convention that $\Sym_p^0=\R$.
\item Let $X$ be a smooth manifold. For each $k\geq 0$ we denote by $\calJ^k(X)$ the $k$-th jet space of mappings from $X$ to 
$\R$, that is, the space $J^k(X,\R)$ introduced in Definition 2.1 of Chapter II of \cite{gg_stable}. For any $p\in\N$ and any open subset $V\subset\R^p$, the space $\calJ^k(V)$ is canonically isomorphic to $V\times\bigoplus_{j=0}^k\Sym_p^j$. We will denote its elements by $(\xi,\omega)$ where $\xi\in V$ and $\omega=(\omega_0,\dots,\omega_k)\in\bigoplus_{j=0}^k\Sym_p^j$.
\item Let $X$ be a smooth manifold and $k\in\N$. For each $f\in C^\infty(X)$, we write $j^kf$ for the section of $\calJ^k(X)$ whose value at each point is the $k$-jet of $f$ at this point (see the paragraph below Definition 2.1 of Chapter II of \cite{gg_stable}).
\end{enumerate}
\end{notation}

Since the jet bundle $\calJ^k\left(\Rnzls\right)$ is quite explicit, we will make most of our contructions inside it and them 'push them down' onto the sphere. In the following lemma, we build the map we need to 'push down' our constructions.
\begin{lemma}\label{l.jet_restriction_map}
 Let $\iota :S^{n-1}\rightarrow\R^n$ be the canonical injection. Then, there exists a bundle morphism
\[
\rho:\iota^*\calJ^k\left(\Rnzls\right)\rightarrow\calJ^k\left(S^{n-1}\right)
\]
such that the following diagram commutes:
\[
\begin{tikzcd}
C^\infty\left(\Rnzls\right)\arrow{r}{\iota^*}  \arrow[swap]{d}{\iota^*\left(j^k\cdot\right)} & C^\infty\left(S^{n-1}\right) \arrow{d}{j^k}\\
\iota^*\calJ^k\left(\Rnzls\right)\arrow{r}{\rho} & \calJ^k\left(S^{n-1}\right)\, .
\end{tikzcd}
\]
Here the top arrow is the restriction map while the left arrow is the restriction of the $k$-jet to the sphere.
\end{lemma}
\begin{prf}
We construct $\rho$ by defining its action on each fiber. Let $\xi\in S^{n-1}$ and let $(V,\phi)$ be a chart $\phi:V\rightarrow\R^{n-1}$ of $S^{n-1}$ near $\xi$. Then, for each $f\in C^\infty\left(\R^n\right)$, the $k$-th order Taylor expansion of $f\circ\phi^{-1}$ at $\xi$ depends only on the $k$-th order Taylor expansion of $f$ at $\xi$ and the dependence is linear. This defines a linear map $\rho|_\xi:\iota^*\calJ^k\left(\Rnzls\right)|_\xi\rightarrow\calJ^k\left(S^{n-1}\right)|_\xi$. The corresponding fiberwise map $\rho$ is clearly smooth and defines a morphism of smooth vector bundles. Moreover, by construction, for each $f\in C^\infty\left(\Rnzls\right)$ and each $\xi\in S^{n-1}$, $\rho|_\xi\left(j^kf(\xi)\right)=j^k(f\circ\iota)(\xi)$ so the diagram does indeed commute.
\end{prf}

\begin{notation}
For each $k\in\N$, each $\xi\in\R^n$ and each $\omega=(\omega_0,\dots,\omega_k)\in\bigoplus_{j=0}^k\Sym_n^j$ we introduce the following notation. For each $j\in\{0,\dots,k\}$, $\omega_j|_{\xi^\perp}$ is the restriction of $\omega_j$ to the orthogonal of $\xi$ in $\R^n$. Moreover, we set $\omega|_{\xi^\perp}=(\omega_0|_{\xi^\perp},\dots,\omega_k|_{\xi^\perp})$.\\
\end{notation}

In the following lemma, we check that the set of jets of homogeneous maps is a smooth submanifold of $\iota^*\calJ^k\left(\Rnzls\right)$ and give an explicit description of it. Moreover, we show that the 'push down' map $\rho$ maps it diffeomorphically on the space $\calJ^k\left(S^{n-1}\right)$.

\begin{lemma}\label{l.jets_of_homogeneous_functions}
Fix $k\in\N$. Let $H^k_m$ be the subset of $\iota^*\calJ^k\left(\Rnzls\right)$ of jets of $m$-homogeneous functions. Then,
\begin{enumerate}
\item The set $H^k_m$ is characterized by the following equations:
\[
H^k_m=\cap_{j=0}^{k-1}\left\{(\xi,\omega)\in\iota^*\calJ^k\left(\Rnzls\right)\, \Big |\, \omega_{j+1}(\xi,\dots)=(m-j)\omega_j\right\}\, .
\]
\item The set $H^k_m$ is a submanifold of $\iota^*\calJ^k\left(U\times\Rnzls\right)$ of the same dimension as $\calJ^k\left(S^{n-1}\right)$.
\item The map $\rho|_{H^k_m}:H^k_m\rightarrow\calJ^k\left(U\times S^{n-1}\right)$ is a diffeomorphism.
\end{enumerate}
\end{lemma}
\begin{prf}
We set
\[
\widetilde{H^k_m}=\cap_{j=0}^{k-1}\left\{(\xi,\omega)\in\iota^*\calJ^k\left(\Rnzls\right)\, \Big |\, \omega_{j+1}(\xi,\dots)=(m-j)\omega_j\right\}\, .
\]
Firstly, each $m$-homogeneous $f\in C^\infty\left(\Rnzls\right)$, satisfies Euler's equation. That is, for each $\xi\in\Rnzls$, $d_\xi f(\xi)=m f(\xi)$. Next, notice that if $f$ is $m$-homogeneous, then, for each $j\in\{1,\dots,k\}$, $\xi\mapsto d^j_\xi f$ is homogeneous of order $m-j$ so that for each $\xi\in\Rnzls$, $d_\xi\left(d^jf\right)(\xi,\dots)=(m-j)d^j_\xi f$. Therefore, for each $\xi\in\Rnzls$, $j^kf(\xi)\in \widetilde{H^k_m}$. We have shown that $H^k_m\subset\widetilde{H^k_m}$.
Next, notice that for each $f\in C^\infty\left(S^{n-1}\right)$, the $m$-homogeneous function $\xi\mapsto |\xi|^mf\left(\frac{\xi}{|\xi|}\right)$ restricts back to $f$ on $S^{n-1}$. Therefore, we have $\calJ^k\left(S^{n-1}\right)=\rho\left(H^k_m\right)\subset\rho\left(\widetilde{H^k_m}\right)\subset\calJ^k\left(S^{n-1}\right)$. So we have
\begin{equation}\label{e.johf.1}
\rho\left(H^k_m\right)=\rho\left(\widetilde{H^k_m}\right)=\calJ^k\left(S^{n-1}\right)\, .
\end{equation}
Given this equation, in order to prove the lemma, it is enough to prove points 2 and 3 with $H^k_m$ replaced by $\widetilde{H^k_m}$, which we call 2' and 3' respectively. Indeed, point 3' will imply that $\rho|_{\widetilde{H^k_m}}$ is one-to-one so by Equation \eqref{e.johf.1}, we will have $H^k_m=\widetilde{H^k_m}$ which is point 1. Moreover, since we will have already proved points 2 and 3 for $\widetilde{H^k_m}$ we will have them for $H^k_m$. Let us start by proving 2'. For each $j\in\{0,\dots,k-1\}$ set
\[
F^j_m:(\xi,\omega)\mapsto \omega_{j+1}(\xi,\dots)-(m-j)\omega_j
\]
so that $\widetilde{H^k_m}=\cap_{j=0}^{k-1} \left(F^j_m\right)^{-1}\left(0\right)$. Let us prove that the map
\[
F_m=\left(F_m^0,\dots,F_m^{k-1}\right):\iota^*\calJ^k\left(\Rnzls\right)\rightarrow\bigoplus_{j=0}^{k-1}\Sym_n^j
\]
is a submersion. Fix $(\xi,\omega)\in\iota^*\calJ^k\left(\Rnzls\right)$. Let $(\eta_0,\dots,\eta_{k-1})\in \bigoplus_{j=0}^{k-1}\Sym_n^j\simeq T_{F_m(\xi,\omega)}\bigoplus_{j=0}^{k-1}\Sym_n^j$. Then, for each $j\in\{0,\dots,k-1\}$,
\[
\partial_{\omega_{j+1}}F_m^j(\xi,\omega)\left(|\xi|^{-2}\langle \xi,\cdot\rangle\otimes\eta_j\right)=\eta_j\, .
\]
In particular, $d_{(\xi,\omega)}F_m$ is surjective. Therefore $\widetilde{H^k_m}$ is a submanifold of $\iota^*\calJ^k\left(\Rnzls\right)$ of codimension
\[
codim_{\iota^*\calJ^k\left(\Rnzls\right)}\left(\widetilde{H^k_m}\right)=\sum_{j=0}^{k-1}dim\left(\Sym_n^j\right)=\sum_{j=0}^{k-1}\binom{n+j-1}{j}\, .
\]
Indeed, recall that $dim\left(\Sym_n^j\right)=\binom{n+j-1}{j}$. Using this identity, we also have:
\begin{align*}
dim\left(\iota^*\calJ^k\left(\Rnzls\right)\right)&=(n-1)+\sum_{j=0}^k\binom{n+j-1}{j}\, ;\\
dim\left(\calJ^k\left(S^{n-1}\right)\right)&=(n-1)+\sum_{j=0}^k\binom{n+j-2}{j}\, .
\end{align*}
Therefore, firstly $dim\left(\widetilde{H^k_m}\right)=(n-1)+\binom{n+k-1}{k}$ and secondly
\begin{equation}\label{e.johf.2}
dim\left(\widetilde{H^k_m}\right)-dim\left(\calJ^k\left(S^{n-1}\right)\right)=\binom{n+k-1}{k}-\sum_{j=0}^{k}\binom{n+j-2}{j}=0\, .
\end{equation}
In the last equality we use a well known binomial formula which is easily checked by induction on $k$. The conclusion here is that $\widetilde{H^k_m}$ has the same dimension as $\calJ^k\left(S^{n-1}\right)$ so we have proved 2'. To prove 3' observe that $\rho$ is linear on each fiber of $\iota^*\calJ^k\left(\Rnzls\right)$ so that its derivative $d\rho$ is constant on each fiber. Moreover, it is equivariant with respect to the automorphisms of the base space $S^{n-1}$ so its derivative must have the same rank on different fibers. Since $\rho$ is surjective (see Equation \eqref{e.johf.1}) $d\rho$ must be of maximal rank. This proves that $\rho$ is a local diffeomorphism. But since it is a morphism of vector bundles, it must be a diffeomorphism, which is the claim of 3'. This concludes the proof of the lemma.
\end{prf}

In the following lemma, we build a submanifold of $H^k_m$ that describes the condition of non-admissibility and compute its codimension.

\begin{lemma}\label{l.jet_equations}
For each $k\in\N$, $k\geq 2$, define
\[
Y^k_m = \cap_{j=2}^k\left\{(\xi,\omega)\in\iota^*\calJ^k\left(\Rnzls\right)\, \Big |\, \omega_0>0,\,  \omega_0^{j-1}\omega_j|_{\xi^\perp}=\frac{m(m-1)\dots(m-j+1)}{m^j}\left(\omega_1|_{\xi^\perp}\right)^{\otimes j}\right\}\, .
\]
Then, $Y^k_m\cap H^k_m$ is a closed submanifold of $H^k_m$ of codimension $\sum_{j=2}^k\binom{n+j-2}{j}$.
\end{lemma}
\begin{prf}
For each $j\in\{0,\dots,k-1\}$, each $l\in\{2,\dots,k\}$ and each $(\xi,\omega)\in\iota^*\calJ^k\left(\Rnzls\right)$, let, as before, $F_m^j(\xi,\omega)=\omega_{j+1}(\xi,\dots)-(m-j)\omega_j\in Sym_n^{j-1}$. Moreover, let $\Sym_n^l|_{\xi^\perp}$ be the set of symmetric $l$-linear forms acting on the orthogonal of $\xi$ in $\R^n$ and let $G_m^l(\xi,\omega)=\omega_l|_{\xi^\perp}-\frac{m(m-1)\dots(m-l+1)}{m^l}\left(\omega_1|_{\xi^\perp}\right)^{\otimes l}\in \Sym_n^l|_{\xi^\perp}$. Then, $Y_m^k\cap H_m^k$ is the intersection of the zero sets of the functions $F_m^j$ and $G_m^l$ for $j\in\{0,\dots,k-1\}$ and $l\in\{2,\dots,k\}$. In particular, it is closed. Note first that $\partial_{\omega_0}F_m^0=m\neq 0\in\textup{Hom}\left(\Sym_n^0,\Sym_n^0\right)\simeq\R$. In particular this map is invertible. We will now prove that for each $l\in\{2,\dots,k\}$, the map $(\partial_{\omega_l} F_m^{l-1},\partial_{\omega_l} G_m^l)$ is of maximal rank on $Y^k_m$. For any $(\xi,\omega)\in Y^k_m$ and any $l\in\{2,\dots,k\}$, $(\partial_{\omega_l}F_m^{l-1}(\xi,\omega),\partial_{\omega_l}G_m^l(\xi,\omega))$ acts as follows.
\begin{align*}
\Sym_n^l&\rightarrow \Sym_n^{l-1}\bigoplus\Sym_n^l|_{\xi^\perp}\\
\eta_l&\mapsto(\eta_l(\xi,\dots),\omega_0^{l-1}\eta_l|_{\xi^\perp})\, .
\end{align*}
But this map is invertible. To see this, let $pr_{\xi^\perp}^*:\Sym_n^l|_{\xi^\perp}\rightarrow\Sym_n^l$ be the pull-back map by the orthogonal projection onto the orthogonal of $\xi$. Also, recall that on $Y^k_m$, we have $\omega_0>0$. Then, the inverse of $(\partial_{\omega_l}F_m^{l-1}(\xi,\omega),\partial_{\omega_l}G_m^l(\xi,\omega))$ is
\begin{align*}
 \Sym_n^{l-1}\bigoplus\Sym_n^l|_{\xi^\perp}&\rightarrow\Sym_n^l\\
(\eta_{l-1},\eta|_{\perp})&\mapsto |\xi|^{-2}\langle \xi,\cdot\rangle\otimes\eta_{l-1}+\omega_0^{1-l}pr_{\xi^\perp}^*\eta_\perp\, .
\end{align*}
All in all, we have shown so far that $\partial_{\omega_0} F_m^0$ is surjective and that for each $l\in\{2,\dots,k\}$, $(\partial_{\omega_l}F_m^{l-1},\partial_{\omega_l}G_m^l)$ is of maximal rank. Therefore, $Y_m^k\cap H^k_m$ is a submanifold of $H^k_m$ of codimension
\begin{align*}
codim_{H^k_m}(Y^k_m\cap H^k_m)&=codim_{\iota^*\calJ^k\left(\Rnzls\right)}\left(Y^k_m\cap H^k_m\right)-codim_{\iota^*\calJ^k\left(\Rnzls\right)}\left(H^k_m\right)\\
&=1+\sum_{l=2}^k\binom{n+l-1}{l}-\sum_{j=0}^{k-1}\binom{n+j-1}{j}\\
&=\binom{n+k-1}{k}-\binom{n+1-1}{1}\\
&=\sum_{j=2}^k\binom{n+j-2}{j}
\end{align*}
where in the last line we use the same binomial identity as in Equation \eqref{e.johf.2}.
\end{prf}

So far we have neglected the $U$ coordinate in the product $U\times S^{n-1}$. To take this coordinate into account, in the following lemma, we introduce a submersion $pr_2:\calJ^k\left(U\times S^{n-1}\right)\rightarrow \calJ^k\left(S^{n-1}\right)$ by which we will pull back the submanifold $\rho\left(Y^k_m\right)$.
\begin{lemma}\label{l.jets_back_home}
Let $k\in\N$. Let $\pi: U\times S^{n-1}\rightarrow S^{n-1}$ be the map $(x,\xi)\mapsto \xi$. Also, for each $x\in U$, let $\iota_x:S^{n-1}\rightarrow U\times S^{n-1}$ be the map $\xi\mapsto (x,\xi)$. Then, there exists a surjective vector bundle morphism $pr_2:\calJ^k\left(U\times S^{n-1}\right)\rightarrow\pi^*\calJ^k\left(S^{n-1}\right)$ such that for each $x\in U$, the following diagram commutes:
\[
\begin{tikzcd}
C^\infty\left(U\times S^{n-1}\right)\arrow{rr}{\iota_x^*} \arrow{d}{j^k} & & C^\infty\left(S^{n-1}\right) \arrow{d}{j^k}\\
\calJ^k\left(U\times S^{n-1}\right)\arrow{r}{pr_2} & \pi^*\calJ^k\left(S^{n-1}\right)\arrow{r}{\iota_x^*} & \calJ^k\left(S^{n-1}\right)\, .
\end{tikzcd}
\]
In particular, $pr_2$ is a submersion.
\end{lemma}
\begin{proof}
Given $f\in C^\infty\left(U\times S^{n-1}\right)$ and $x\in U$, the $k$-jet of $f(x,\cdot)$ at $\xi\in S^{n-1}$ depends only on the $k$-jet of $f$ at $(x,\xi)$. This allows us to define a map $pr_2|_{(x,\xi)}:\calJ^k\left(U\times S^{n-1}\right)|_{(x,\xi)}\rightarrow\pi^*\calJ^k\left(S^{n-1}\right)_{(x,\xi)}$. This defines a bundle morphism $pr_2:\calJ^k\left(U\times S^{n-1}\right)\rightarrow\pi^*\calJ^k\left(\right)$. The fact that the diagram commutes follows by construction. Finally, since the composition of the top and right arrows : $j^k\circ\iota_x^*$ is onto, so is the composition of the left and bottom arrows. But this implies that the composition of bottom arrows is onto. Since $\pi^*\calJ^k\left(S^{n-1}\right)$ and $\calJ^k\left(S^{n-1}\right)$ have the same rank, then $pr_2$ must also be onto. In particular, it defines a submersion from the manifold $\calJ^k\left(U\times S^{n-1}\right)$ to the manifold $\pi^*\calJ^k\left(S^{n-1}\right)$.
\end{proof}

In this last lemma, we check that the previous construction does indeed characterize non-admissibility of a symbol by the intersection of the $k$-jet with the submanifold constructed in Lemma \ref{l.jet_equations} and 'pushed down' by $\rho$.

\begin{lemma}\label{l.jets_we_want}
Let $k\in\N$, $k\geq 2$. Let $\sigma\in S^m_{h,+}(U)$. Then, there exists $(x,\xi)\in U\times(\Rnzls)$ such that for each $j\in\{2,\dots,k\}$
\begin{equation}\label{e.jww.1}
\sigma^{j-1}(x,\xi)\partial_\xi^j\sigma(x,\xi)=\frac{m(m-1)\dots(m-j+1)}{m^j}\left(\partial_\xi\sigma(x,\xi)\right)^{\otimes j}
\end{equation}
if and only if $pr_2\circ j^k\left(\sigma|_{U\times S^{n-1}}\right)\left(U\times S^{n-1}\right)\cap \rho\left(Y^k_m\right)\neq\emptyset$.
\end{lemma}
\begin{prf}
Firstly, Equation \eqref{e.jww.1} is homogeneous in $\xi$ so there exists a pair $(x,\xi)\in U\times(\Rnzls)$ satisfying it if and only if there exists such a pair in $U\times S^{n-1}$. Now, since $\sigma$ is $m$-homogeneous, for each $x\in U$, $j^k(\sigma(x,\cdot))(S^{n-1})\subset H^k_m$. Therefore, $(x,\xi)\in U\times S^{n-1}$ satisfy Equation \eqref{e.jww.1} if and only if $j^k(\sigma(x,\cdot))(\xi)\in Y^k_m\cap H^k_m$ (here we use that the symbols are positive, as well as $m$-homogeneous). Since, moreover, by Lemma \ref{l.jet_equations}, $\rho|_{H^k_m}$ is bijective, this is equivalent to $\rho\circ j^k(\sigma(x,\cdot))(\xi)\in\rho(Y^k_m)$. But, by Lemmas \ref{l.jet_restriction_map} and \ref{l.jets_back_home}, $\rho\circ j^k(\sigma(x,\cdot))=j^k(\sigma(x,\cdot)|_{S^{n-1}})=pr_2\circ j^k\left(\sigma|_{U\times S^{n-1}}\right)(x,\cdot)$. To conclude, we have proved that for any $(x,\xi)\in U\times S^{n-1}$, $(x,\xi)$ satisfies Equation \eqref{e.jww.1} if and only if $pr_2\circ j^k\left(\sigma|_{S^{n-1}}\right)(x,\xi)\in\rho\left(H^k_m\right)$. This concludes the proof of the lemma.
\end{prf}

We are now ready to prove Proposition \ref{p.fat}.

\begin{prf}[ of Proposition \ref{p.fat}]
Firstly, by Lemma \ref{l.jets_we_want}, Equation \eqref{e.fat} has solutions in $U\times(\Rnzls)$ if and only if $j^k\left(\sigma|_{U\times S^{n-1}}\right)(U\times S^{n-1})\cap pr_2^{-1}\left(\rho(Y^k_m)\right)\neq\emptyset$. Now, by Lemmas \ref{l.jets_of_homogeneous_functions} and \ref{l.jet_equations}, $\rho\left(Y^k_m\right)$ is a closed submanifold of $\calJ^k\left(S^{n-1}\right)$ of codimension $\sum_{j=2}^k\binom{n+j-2}{j}$. Since moreover, by Lemma \ref{l.jets_back_home}, $pr_2$ is a submersion, $Z^k_m=pr_2^{-1}\left(\rho(Y^k_m)\right)$ has the same codimension in $\calJ^k\left(U\times S^{n-1}\right)$. At this point, we apply Thom's transversality theorem (Corollary 4.10 of Chapter II of \cite{gg_stable}). This theorem states that the functions $f\in C^\infty\left(U\times S^{n-1}\right)$ such that $j^k(f)(U\times S^{n-1})$ is transverse to $Z^k_m$ is open and dense. But $j^k(f)\left(U\times S^{n-1}\right)$ has dimension at most $2n-1$ so if $k$ is such that
\begin{equation}\label{e.fat.1}
2n-1<\sum_{j=2}^k\binom{n+j-2}{j}
\end{equation}
then such a transverse intersection must be empty. Inequality \eqref{e.fat.1} is satisfied for instance for $n=2$ and $k=5$, for $n\in\{3,4\}$ and $k=3$ and for $n\geq 5$ and $k=2$. This ends the proof of the proposition.
\end{prf}

\appendix

\section{Proof of Theorem \ref{t.hormander}}\label{s.hormander}

In this section, we prove Theorem \ref{t.hormander} by following closely the approach used in \cite{ho68} and in \cite{gawe14}. As explained above, \cite{gawe14} contains all the essential arguments for Theorem \ref{t.hormander} despite the focus on the case where $x=y$ and $\man$ is closed. In this section we merely wish to confirm this by revisiting the proof. We consider $A$, $\sigma_A$ and $E_L$ indifferently as in any of the two settings presented in Subection \ref{ss.setting}.

\subsection{Preliminaries}

The following lemma summarizes the results proved in Section 4 of \cite{ho68} for the closed manifold setting. For the boundary problem, this was proved in Section 3 of \cite{vas83}. We introduce the following notation. For each $L>0$, set $\tilde{E}_L=E_{L^m}$.

\begin{lemma}\label{l.source}
Firstly, the spectral function $\tilde{E}_L(x,y)$ defines a tempered distribution of the $L$ variable with values in $C^\infty(\man\times \man)$. In addition, for each set of local coordinates in which $d\mu_{\man}$ coincides with the Lebesgue measure on $\R^n$, there is an open neighborhood $U$ of $0\in\R^n$ such that there exist $\eps>0$, a proper phase function $\psi \in C^\infty(U\times U\times\R^n)$, a symbol $\sigma\in S^1(U,\R^n)$, a function $k\in C^\infty(U\times U\times]-\eps,\eps[)$ and a symbol $q\in S^0(U\times ]-\eps,\eps[\times U,\R^n)$, for which
\[\fou_L[\tilde{E}_L'(x,y)](t)=\frac{1}{(2\pi)^n}\int_{\R^n}q(x,t,y,\xi)e^{i(\psi(x,y,\xi)-t\sigma(y,\xi))}d\xi + k(x,y,t).\]
Here $\fou_L$ (resp. $'$) denotes the Fourier transform (resp. the derivative) with respect to the variable $L$, in the sense of temperate distributions, and the integral is to be understood in the sense of Fourier integral operators (see Theorem 2.4 of \cite{ho68}). We have
\begin{enumerate}
\item The function $\psi$ satisfies the Equation
\[\forall x,y\in U,\ \xi\in \R^n,\ \sigma(x,\partial_x\psi(x,y,\xi))=\sigma(y,\xi).\]
\item For each $t\in]-\eps,\eps[$ and $\xi\in\R^n$, the function $q(\cdot,t,\cdot,\xi)$ has compact support in $U\times U$ uniformly in $(t,\xi)$ and $q(x,0,y,\xi)-1$ is a symbol of order $-1$ as long as $x,y$ belong to some open neighborhood $U_0$ of $0$ in $U$.
\item $\sigma-\sigma_A^{\frac{1}{m}}\in S^0$.
\end{enumerate}
\end{lemma}
We will also need the following classical lemma. Here and below, $\mathcal{S}(\R)$ will denote the space of Schwartz functions.
\begin{lemma}\label{bump}
For each $\eps>0$ there is a function $\rho\in\mathcal{S}(\R)$ such that $\fou(\rho)$ has compact support contained in $]-\eps,\eps[$, $\rho>0$ and $\fou(\rho)(0)=1$.
\end{lemma}
\begin{prf}
Choose $f\in\mathcal{S}(\R)$ whose Fourier transform has support in $]-\frac{\eps}{2},\frac{\eps}{2}[$. Then it is easy to see that $\rho=f^2*f^2$ satisfies the required properties.
\end{prf}
Before we proceed, let us fix $U$, $\psi$, $q$, $k$ and $\rho$ as in Lemmas \ref{l.source} and \ref{bump} as well as a differential operator $P$ on $\man\times \man$ of order $d$ with principal symbol $\sigma_P$. Let $\tilde{E}_{L,P}=P\tilde{E}_L$. In order to estimate this $\tilde{E}_{L,P}$, we will first convolve it with $\rho$ in order to estimate it using Lemma \ref{l.source}. Then, we will compare $\tilde{E}_{L,P}$ to its convolution with $\rho$ which we denote - somewhat liberally - by
\[
\rho*\tilde{E}_{L,P}=\int_{\R}\rho(\lambda)\tilde{E}_{L-\lambda,P}d\lambda\, .
\]
The starting point of the following calculations will be the following Equation, which follows from Lemma \ref{l.source}.
\begin{align}\label{e.ap}
\frac{d}{d\lambda}(\rho*e_{\lambda,P}(x,y))|_{\lambda=L}=\frac{1}{(2\pi)^n}\int_{\R^n}\fou^{-1}_t\Big[\fou(\rho)(t)&P \left(q(x,t,y,\xi)e^{i(\psi(x,t,y,\xi)-t\sigma(y,\xi))}\right) \Big](L)d\xi\\\nonumber
&+\fou^{-1}_t\big[\fou(\rho)(t)Pk(x,t,y)\big](L). 
\end{align}

\subsection{Estimating the convolved kernel}

In this section we provide the following expression for $\rho*\tilde{E}_{L,P}$ in the local coordinates chosen in Lemma \ref{l.source}.
\begin{lemma}\label{hol0}
There is an open set $V\subset U$ containing $0$ such that, as $L\rightarrow\infty$ and uniformly for $(x,y)\in V\times V$,
\[\rho*\tilde{E}_{L,P}(x,y) = \frac{1}{(2\pi)^n}\int_{\sigma(y,\xi)\leq L} \sigma_P(x,y,\partial_{x,y}\psi(x,y,\xi)) e^{i\psi(x,y,\xi)}d\xi + O(L^{n+d-1})\,.\]
\end{lemma}
In order to do so we use the three lemmas stated below, whose proofs are given at the end of the section. To begin with, we use the information of Lemma \ref{l.source} to give a first expression for $\rho*\tilde{E}_{L,P}$.
\begin{lemma}\label{hol1}
The quantity
\[\rho*\tilde{E}_{L,P}(x,y)-\int_{-\infty}^L\frac{1}{(2\pi)^n}\int_{T^*_yM} \fou^{-1}_t\Big[\fou(\rho)P\left(q(x,t,y,\xi)e^{i(\psi(x,y,t,\xi)-t \sigma(y,\xi))}\right)\Big](\lambda)d\xi d\lambda\]
is bounded uniformly for $(x,y)\in U\times U$.
\end{lemma}
Here and below $\fou$ is the Fourier transform and the occasional subscript indicates the variable on which the transform is taken. Let us now investigate the effect of the differential operator $P$ on the right hand side of this expression. By the Leibniz rule, there is a finite family of symbols $(\sigma_j)_{0\leq j\leq d}\in C^{\infty}(U\times]-\eps,\eps[\times U,\R^n)^{d+1}$ such that for each $j$, $\sigma_j$ is homogeneous of degree $j$, such that
\[P\Big[q(x,t,y,\xi)e^{i(\psi(x,y,\xi)-t\sigma(y,\xi))}\Big]=\Big[\sum_{j=0}^{d} \sigma_j(x,t,y,\xi)\Big]e^{i(\psi(x,y,\xi)-t\sigma(y,\xi))}\]
and such that
\[\sigma_d(x,t,y,\xi)=q(x,t,y,\xi)\sigma_P(x,y,\partial_{x,y}(\psi(x,y,\xi)-t\sigma(y,\xi)))\,.\]
Now, for each $j$, let
\[R_j(x,y,L,\xi)=\frac{1}{(2\pi)^{n+1}}\int_{\R}\fou(\rho)(t)\sigma_j(x,t,y,\xi)e^{itL}dt\]
and
\[S_j(x,y,L)=\int_{-\infty}^L\int_{\R^n} R_j(x,y,\lambda-\sigma(y,\xi),\xi)e^{i\psi(x,y,\xi)}d\xi d\lambda\,.\]
Then,
\[\int_{-\infty}^L\frac{1}{(2\pi)^n}\int_{\R^n} \fou^{-1}_t\Big[\fou(\rho)P\left(q(x,t,y,\xi)e^{i(\psi(x,y,\xi)-t \sigma(y,\xi))}\right)\Big](\lambda)d\xi d\lambda=\sum_{j=0}^dS_j(x,y,L)\,.\]
Each $S_j$ will grow at an order corresponding to the degree of the associated symbol. This is shown in the following lemma.
\begin{lemma}\label{hol3}
There is an open set $V\subset U$ containing $0$ such that, as $L\rightarrow\infty$ and uniformly for $(x,y)\in V\times V$,
\[S_j(x,y,L)=\frac{1}{(2\pi)^n}\int_{\sigma(y,\xi)\leq L}\sigma_j(x,0,y,\xi)e^{i\psi(x,y,\xi)}d\xi+O(L^{n+j-1})\,.\]
\end{lemma}
Similarly since $q(x,0,y,\xi)-1\in S^{-1}(U_0\times U_0,\R^n)$, from a computation analogous to the proof of Lemma \ref{hol3} and left to the reader, replacing $\sigma_d$ by
\[(q(x,0,y,\xi)-1)\sigma_P(x,y,\partial_{x,y}(\psi(x,y,\xi)-t\sigma(y,\xi)))\in S^{d-1}\]
one can remove $q$ from the main term, which results in the following.
\begin{lemma}\label{hol4}
There is an open set $V\subset U$ containing $0$ such that, as $L\rightarrow\infty$ and uniformly for $(x,y)\in V\times V$,
\[
S_d(x,y,L)=\frac{1}{(2\pi)^n}\int_{\sigma(y,\xi)\leq L} \sigma_P(x,y,\partial_{x,y}(\psi(x,y,\xi)-t\sigma(y,\xi))) e^{i\psi(x,y,\xi)}d\xi + O(L^{n+d-1})\,.
\]
\end{lemma}
The juxtaposition of these results yields Lemma \ref{hol0}.
\begin{prf}[ of Lemma \ref{hol1}]
Since $k\in C^\infty(U\times U\times]-\eps,\eps[)$ and $\fou(\rho)$ is supported in $]-\eps,\eps[$,
\[
\fou^{-1}_t\big[\fou(\rho)(t)Pk(x,t,y)\big](L)\in\mathcal{S}(\R)\, .
\]
Therefore, by Equation \eqref{e.ap},
\[
\rho*\tilde{E}_{L,P}(x,y)-\int_{-\infty}^L\frac{1}{(2\pi)^n}\int_{\R^n}\fou^{-1}_t\Big[\fou(\rho)P\left(q(x,t,y,\xi)e^{i(\psi(x,y,\xi)-t \sigma(y,\xi))}\right)\Big](\lambda)d\xi d\lambda
\]
is bounded.
\end{prf}
\begin{prf}[ of Lemma \ref{hol3}]
In this proof, all generic constants will be implicitly uniform with respect to $(x,y)\in V\times V$. Let us fix $y\in V$ and define the following three domains of integration.
\begin{align*}
D_1&=\{(\lambda,\xi)\in\R\times \R^n\ |\ \lambda\leq L,\ \sigma(y,\xi)\leq L\}\\
D_2&=\{(\lambda,\xi)\in\R\times \R^n\ |\ \lambda\leq L,\ \sigma(y,\xi)>L\}\\
D_3&=\{(\lambda,\xi)\in\R\times \R^n\ |\ \lambda>L,\ \sigma(y,\xi)\leq L\}\, .
\end{align*}
Moreover, for $l=1,2,3$, let $I_l=\int_{D_l}R_j(x,y,\lambda-\sigma(y,\xi),\xi)e^{i\psi(x,y,\xi)}d\xi d\lambda$. We will prove that $I_2$ and $I_3$ are $O(L^{n+j-1})$. The following calculation will then yield the desired identity. Here we use Fubini's theorem and the fact that $\fou(\rho)(0)=\int_{\R}\rho(\lambda)d\lambda=1$.
\begin{align*}
S_j(x,y,L)&= I_1+I_2= I_1+I_3+O(L^{n+j-1})\\
          &= \int_{\sigma(y,\xi)\leq L}\Big[\int_{\R}R_j(x,y,s,\xi)ds\Big]e^{i\psi(x,y,\xi)}d\xi+O(L^{n+j-1})\\
          &= \frac{1}{(2\pi)^n}\int_{\sigma(y,\xi)\leq L}\sigma_j(x,0,y,\xi)e^{i\psi(x,y,\xi)}d\xi+O(L^{n+j-1}).
\end{align*}
First of all, $R_j$ is rapidly decreasing in the third variable and, since $\sigma$ is elliptic of degree $1$, bounded by $\sigma(y,\xi)^j$ with respect to the last variable, $\xi$. Therefore, for each $N>0$ there is a constant $C>0$ such that
\[|R_j(x,y,\lambda,\xi)|\leq \frac{C \sigma(y,\xi)^j}{(1+|\lambda|)^N}\,.\]
Since $\sigma$ is elliptic of order $1$, the hypersurface $L^{-1}\{\sigma(y,\xi)=L\}\subset \R^n$ converges smoothly for $L\rightarrow\infty$ uniformly in $y$ to $S^*_y=\{\sigma_A(y,\xi)=1\}$ and the volume of $\{\sigma(x,\xi)= \beta\}\subset \R^n$ is $O(\beta^{n-1})$. Taking $N= 2n+j+1$, we deduce that
\begin{align*}
|I_2|&\leq C\int_{-\infty}^L\int_{\sigma(y,\xi)>L}\frac{\sigma(y,\xi)^j}{(1+|\lambda-\sigma(y,\xi)|)^{2n+j+1}}d\xi d\lambda\leq C\int_{-\infty}^L\int_L^{+\infty}\frac{\beta^{n+j-1}}{(1+|\lambda-\beta|)^{2n+j+1}}d\beta d\lambda\\
     &\leq C\int_L^{+\infty} \int_{-\infty}^{L-\beta}\frac{\beta^{n+j-1}}{(1+|s|)^{2n+j+1}}dsd\beta\leq C\int_L^\infty \frac{\beta^{n+j-1}}{(1+\beta-L)^{2n+j}}d\beta\\
     &\leq C\int_0^{+\infty} \frac{(\gamma+L)^{n+j-1}}{(1+\gamma)^{2n+j}}d\gamma\leq C L^{n+j-1}.
\end{align*}
Here we applied first the change of variables $s=\lambda-\beta$ and then $\gamma=\beta-L$. The case of $I_3$ is analogous and by a similar calculation we deduce that $I_1$ is well defined.
\end{prf}

\subsection{Comparison of the kernel and its convolution}

In this section we set about proving that $\tilde{E}_{L,P}$ is close enough to its convolution with $\rho$. This is encapsulated in the following lemma.
\begin{lemma}\label{hol6}
There is an open set $V\subset U$ containing $0$ such that, as $L\rightarrow\infty$ and uniformly for $(x,y)\in V\times V$,
\[\rho*\tilde{E}_{L,P}(x,y)-\tilde{E}_{L,P}(x,y)=O(L^{n+d-1})\,.\]
\end{lemma}
As before, the proofs are relegated to the end of the section. In order to prove Lemma \ref{hol6} we first estimate the growth of the $R_j$ as follows.
\begin{lemma}\label{hol2}
There is an open set $V\subset U$ containing $0$ such that, as $L\rightarrow\infty$ and uniformly for $(x,y)\in V\times V$,
\[\int_{\R^n} R_j(x,y,L-\sigma(y,\xi),\xi)e^{i\psi(x,y,\xi)}d\xi=O(L^{n+j-1})\,.\] 
\end{lemma}
This lemma follows from a computation analogous to the bound on $I_2$ and $I_3$ given in the proof of Lemma \ref{hol3} above and the details are left to the reader. It allows us to prove a second intermediate result from which we obtain Lemma \ref{hol6} directly.
\begin{lemma}\label{hol5}
There is an open set $V\subset U$ containing $0$ such that, as $L\rightarrow\infty$ and uniformly for $(x,y)\in V\times V$,
\[\tilde{E}_{L+1,P}(x,y)-\tilde{E}_{L,P}(x,y)=O(L^{n+d-1})\,.\]
\end{lemma}
\begin{prf}[ of Lemma \ref{hol5}]
We begin with the case where $x=y$ and $P$ is of the form $P_1\otimes P_1$. For brevity we define
\[u(L)=\tilde{E}_{L,P}(x,x)=\sum_{\lambda_k\leq L}|(P_1 e_k)(x)|^2\,.\]
Recall $\rho>0$ so it stays greater than some constant $a>0$ on the interval $[-1,0]$. Moreover $u$ is increasing so by Equation \eqref{e.ap} and Lemma \ref{hol2},
\begin{align*}
0&\leq u(L+1)-u(L)=\int_L^{L+1}u'(\lambda)d\lambda\leq \frac{1}{a}\int_L^{L+1}\rho(L-\lambda)u'(\lambda)d\lambda\\
 &\leq \frac{1}{a}\frac{d}{dL}(\rho*u)\leq \frac{1}{a}\sum_{j=0}^d \int_{\R^n} R_j(x,y,L-\sigma(y,\xi))d\xi +O(L^{n+d-1})=O(L^{n+d-1}).
\end{align*}
Now if $P$ is of the form $P_1\otimes P_2$, and for any $x$ and $y$, let $X=(P_1 e_k)_{L<\lambda_k\leq L+1}$ and $Y=(P_2 e_k)_{L<\lambda_k\leq L+1}$ be two vectors in some $\C^q$ which we equip with the standard hermitian product ``$\star$''. Then, $\tilde{E}_{L+1,P}(x,y)-\tilde{E}_{L,P}(x,y)=X\star \overline{Y}$ so
\begin{align*}
|\tilde{E}_{L+1,P}&(x,y)-\tilde{E}_{L,P}(x,y)|^2\leq |X|^2|Y|^2\\
                   &=|\tilde{E}_{L+1,P_1\otimes P_1}(x,y)-\tilde{E}_{L,P_1\otimes P_1}(x,y)||\tilde{E}_{L+1,P_2\otimes P_2}(x,y)-\tilde{E}_{L,P_2\otimes P_2}(x,y)|\\
                   &\leq \frac{1}{4}\left(\tilde{E}_{L+1,P_1\otimes P_1}(x,x)-\tilde{E}_{L,P_1\otimes P_1}(x,x)+\tilde{E}_{L+1,P_1\otimes P_1}(y,y)-\tilde{E}_{L,P_1\otimes P_1}(y,y)\right)\\
                   &\times\left(\tilde{E}_{L+1,P_2\otimes P_2}(x,x)-\tilde{E}_{L,P_2\otimes P_2}(x,x)+\tilde{E}_{L+1,P_2\otimes P_2}(y,y)-\tilde{E}_{L,P_2\otimes P_2}(y,y)\right)\\
                   &\leq C L^{2n+2d-2}.
\end{align*}
Here we used first the Cauchy-Schwarz inequality, then the mean value inequality, then on each factor,
\[2|P_1 e_k(x)\overline{P_1 e_k(y)}|\leq |P_1 e_k(x)|^2+|P_1 e_k(y)|^2\]
and finally the above estimate. In general $P$ is a locally finite sum of operators of the form $P_1\otimes P_2$.
\end{prf}
\begin{prf}[ of Lemma \ref{hol6}]
First of all, according to Lemma \ref{hol5} there is a constant $C$ such that for all $L\geq 0$ and $\lambda$,
\begin{equation*}
|\tilde{E}_{L+\lambda,P}(x,y)-\tilde{E}_{L,P}(x,y)|\leq C(1+|\lambda|+L)^{n+d-1}(1+|\lambda|).
\end{equation*}
Consequently
\begin{align*}
(\rho*\tilde{E}_{L,P}(x,y)-\tilde{E}_{L,P}(x,y)|&\leq|\int\rho(\lambda)\tilde{E}_{L+\lambda,P}(x,y)d\lambda-\tilde{E}_{L,P}(x,y)|\\
                                               &\leq\int\rho(\lambda)\Big|\tilde{E}_{L+\lambda,P}(x,y)-\tilde{E}_{L,P}(x,y)\Big|d\lambda\\
                                               &\leq C\int\rho(\lambda)(1+|\lambda|+L)^{n+d-1}(1+|\lambda|)d\lambda\\
                                               &\leq C'L^{n+d-1}
\end{align*}
for some $C'>0$. Here we used that $\rho>0$, $\rho$ is rapidly decreasing and $\int_{\R}\rho(\lambda)d\lambda=\fou(\rho)(0)=1$.
\end{prf}

\subsection{Conclusion}

Combining Lemmas \ref{hol0} and \ref{hol6} we obtain the following:
\[
\tilde{E}_{L,P}(x,y)=\frac{1}{(2\pi)^n}\int_{\sigma(y,\xi)\leq L} \sigma_P(x,y,\partial_{x,y}\psi(x,y,\xi)) e^{i\psi(x,y,\xi)}d\xi + O(L^{n+d-1})\, .
\]
Since $\sigma-\sigma_A^{\frac{1}{m}}\in S^0$, replacing one by the other adds only a $O(L^{n+d-1})$ term. Therefore,
\[
\tilde{E}_{L,P}(x,y)=\frac{1}{(2\pi)^n}\int_{\sigma_A(y,\xi)^{1/m}\leq L} \sigma_P(x,y,\partial_{x,y}\psi(x,y,\xi)) e^{i\psi(x,y,\xi)}d\xi + O(L^{n+d-1})\, .
\]
This estimate is valid and uniform for $x,y\in V$. To conclude, notice that $\sigma_A(x,\xi)^{1/m}\leq L$ is equivalent to $\sigma_A(x,\xi)\leq L^m$. Since $\tilde{E}_L=E_{L^m}$, replacing $L$ by $L^{1/m}$ in the last estimate we get
\[
E_{L,P}(x,y)=\frac{1}{(2\pi)^n}\int_{\sigma_A(y,\xi)
\leq L} \sigma_P(x,y,\partial_{x,y}\psi(x,y,\xi)) e^{i\psi(x,y,\xi)}d\xi + O(L^{(n+d-1)/m})
\]
as announced.

\nocite{*}
\bibliography{bib_43}
\bibliographystyle{alpha}

\end{document}